\documentclass[12pt]{amsart}

\usepackage[english]{babel}

\usepackage[letterpaper,top=2cm,bottom=2cm,left=3cm,right=3cm,marginparwidth=1.75cm]{geometry}

\usepackage{amsmath}
\usepackage{enumerate}
\usepackage[all]{xy}
\usepackage{graphicx}
\usepackage{tikz}
\usepackage[colorlinks=true, allcolors=blue]{hyperref}

\usepackage{amssymb,amsthm}

\theoremstyle{plain}
\newtheorem{theorem}{Theorem}[section]
\newtheorem{lema}[theorem]{Lemma}
\newtheorem{corolario}[theorem]{Corollary}
\newtheorem{proposition}[theorem]{Proposition}

\theoremstyle{definition}
\newtheorem{definition}[theorem]{Definition}
\newtheorem{example}[theorem]{Example}
\theoremstyle{remark}
\newtheorem{obs}[theorem]{Remark}
\newtheorem*{theorem*}{Theorem}

\theoremstyle{plain}

\newcommand{\id}{\operatorname{id}}

\DeclareMathOperator{\Supp}{Supp}

\title[Prime groupoid graded rings]{Prime groupoid graded rings with applications to partial skew groupoid rings}

\author{Paula S. E. Moreira}
\address{Departamento de Matem\'atica, Universidade Federal de Santa Catarina, 88040-900 \\	Florian\'opolis, Brasil.}
\email{savanamatematica@gmail.com}
\thanks{The first named author was financed in part by the Coordena\c{c}\~{a}o de Aperfei\c{c}oamento de Pessoal de N\'{i}vel Superior - Brasil (CAPES) - Finance Code 001.
The authors are grateful to Patrik Lundstr\"{o}m for making them aware of Munn's result in \cite{Munn2000}.}

\author{Johan \"{O}inert}
\address{Department of Mathematics and Natural Sciences, Blekinge Institute of Technology, SE-37179 Karlskrona, Sweden.}
\email{johan.oinert@bth.se}

\makeatletter
\@namedef{subjclassname@2020}{%
  \textup{2020} Mathematics Subject Classification}
\makeatother



\begin{document}
\maketitle

\begin{abstract}
In this paper, we investigate primeness of groupoid graded rings.
We provide a set of necessary and sufficient conditions for primeness of a nearly-epsilon strongly groupoid graded ring.
Furthermore, we apply our main result to get a characterization of prime partial skew groupoid rings, and in particular of prime groupoid rings, thereby generalizing a classical result by Connell and partially generalizing recent results by Steinberg.
\end{abstract}

\section{Introduction}

Throughout this paper, all rings are assumed to be associative but not necessarily unital.
Recall that a ring $S$ is said to be \emph{prime} if there are no nonzero ideals $I, J$ of $S$ such that $IJ=\{0\}.$

In 1963, Connell \cite[Thm.~8]{Connell1963} gave a characterization of prime group rings.
Indeed, given a unital ring $A$ and a group $G$, the corresponding group ring $A[G]$ is prime if, and only if, $A$ is prime and $G$ has no non-trivial finite normal subgroup.

Given a group $G$,  recall that a ring $S$ is said to be \emph{$G$-graded} if there is a collection $\{S_g\}_{g\in G}$ of additive subgroups of $S$ such that $S=\oplus_{g\in G}S_g$, and $S_g S_h \subseteq S_{gh}$, for all $g,h\in G$.
If, in addition, $S_g S_h = S_{gh}$, for all $g,h\in G$, then $S$ is said to be \emph{strongly $G$-graded}.
The class of unital strongly $G$-graded rings includes for instance all group rings, all twisted group rings, and all $G$-crossed products.

In 1984, Passman \cite[Thm.~1.3]{Passman1984} generalized Connell's result by giving a characterization of prime unital strongly group graded rings.
In recent years, various generalizations of strongly group graded rings have appeared in the literature.
In \cite{Nystedt2020}, so-called
\emph{nearly epsilon-strongly group graded rings} were introduced. That class of rings contains for instance all unital strongly group graded rings, all epsilon-strongly group graded rings, all Leavitt path algebras and all unital partial crossed products (see \cite{Nystedt2020,Pinedo2018}).
In \cite{Lannstrom2021}, a characterization of prime nearly epsilon-strongly group graded rings was established, thereby generalizing Passman's result to a non-unital and non-strong setting.

In this paper we turn our focus to rings graded by groupoids.
Let $G$ be a groupoid. Recall that a ring $S$ is said to be \emph{$G$-graded} if there is a collection $\{S_g\}_{g\in G}$ of additive subgroups  of $S$ such that $S = \oplus_{g\in G} S_g$, and $S_g S_h \subseteq S_{gh}$ whenever $g,h \in G$ are composable, and $S_g S_g = \{0\}$ otherwise.
Natural examples of groupoid graded rings that are not necessarily group graded, are for instance groupoid rings, groupoid crossed products, and partial skew groupoid rings (see e.g. \cite{Bagio2010, Moreira2022, Bagio2021, Bagio2012, Bagio2022,OL2010}).

Suppose that $G$ is a groupoid and that $S$ is a $G$-graded ring.
Following \cite{Lannstrom2022},
we shall say that $S$ is \emph{nearly epsilon-strongly $G$-graded} if, for each $g\in G$,  $S_g S_{g^{-1}}$ is an $s$-unital ring and $S_g S_{g^{-1}} S_g = S_g$.
Our main result provides a characterization of prime nearly epsilon-strongly groupoid graded rings.

\begin{theorem}\label{P2-teo-nesgr} 
Let $G$ be a groupoid, let $S$ be a nearly epsilon-strongly $G$-graded ring,
and let $G':=\{g \in G: S_{s(g)}\neq\{0\} \text{ and } S_{r(g)}\neq\{0\}\}$. 
The following statements are equivalent:
\begin{enumerate}[{\rm (i)}]
    \item $S$ is prime;
    \item	$\oplus_{e \in G_0}S_e$ is $G$-prime, and for every $e \in G_0',$ $\oplus_{g \in G_e^e}S_g$ is prime;
    \item	$\oplus_{e \in G_0}S_e$ is $G$-prime, and for some $e \in G_0',$ $\oplus_{g \in G_e^e}S_g$ is prime;
    \item	$S$ is graded prime, and for every $e \in G_0',$ $\oplus_{g \in G_e^e}S_g$ is prime;
    \item	$S$ is graded prime, and for some $e \in G_0',$  $\oplus_{g \in G_e^e}S_g$ is prime;
    \item	For every $e \in G_0',$ $e$ is a support-hub, and  $\oplus_{g \in G_e^e}S_g$ is prime; 
    \item	For some $e \in G_0',$ $e$ is a support-hub, and  $\oplus_{g \in G_e^e}S_g$ is prime. 
\end{enumerate}
\end{theorem}
\noindent Here $G_e^e$ denotes the isotropy group of an element $e \in G_0'.$ For more details about the statements in the above theorem, see e.g. Definitions~\ref{P2-def-prime-nearly} and \ref{P2:def-support-hub}.

We point out that our main result reduces the primeness investigation for a groupoid graded ring to the group case. Indeed, $\oplus_{g\in G_e^e} S_g$ is a nearly epsilon-strongly group graded ring (cf.~\cite{Nystedt2020}). Hence, the main result of \cite{Lannstrom2021} can be used to decide whether it is prime.

Here is an outline of this paper.
In Section~\ref{P2-Preliminaries}, based on \cite{Nystedt2019}, \cite{Oinert2019} and \cite{Tominaga1976}, we recall some basic definitions and properties about groupoids, groupoid graded rings and $s$-unital rings that will be used throughout the paper.
In Section~\ref{P2-primeness}, we record some basic properties of nearly epsilon-strongly groupoid graded rings. Inspired by \cite{Lannstrom2021}, for such a ring $S$, we establish a relationship between the $G$-invariant ideals of $\oplus_{e \in G_0}S_e$ and the $G$-graded ideals of $S$ (see Theorem~\ref{P2-thm:InvGradedIdealCorresp}). Moreover, we provide necessary conditions for graded primeness of $S$ (see Section~\ref{graded-prime}) and establish our main result which is a characterization of prime nearly epsilon-strongly groupoid graded rings (see Theorem~\ref{P2-teo-nesgr}).
Finally, in Section~\ref{P2-Applications}, we apply our results to partial skew groupoid rings, skew groupoid rings and groupoid rings.  
Indeed, 
we give a characterization of prime partial skew groupoid rings associated with groupoid partial actions of group-type \cite{Bagio2021,Bagio2022} (see Theorem~\ref{P2-prop:g-type-prime}).
Using that every global action of a connected groupoid is of group-type, we get a characterization of prime skew groupoid rings of connected groupoids (see Theorem~\ref{P2:teo-global-connected}).  
Furthermore, we establish a generalization of Connell's classical result, by providing a characterization of prime groupoid rings (see Theorem~\ref{P2-teogr}).

\section{Preliminaries}\label{P2-Preliminaries}

In this section, we recall some  notions and basic notation regarding groupoids and graded rings.

\subsection{Groupoids}

By a \emph{groupoid} we shall mean a small category $G$ in which every morphism is invertible.
Each object of $G$ will be identified with its corresponding identity morphism, allowing us to view $G_0$, the set of objects of $G$, as a subset of the set of morphisms of $G$.
The set of morphisms of $G$ will simply be denoted by $G$. This means that
$G_0:=\{gg^{-1} : g \in G\} \subseteq G$.

The \emph{range} and \emph{source} maps
$r,s : G \to G_0$,
indicate the range (codomain)
respectively
source (domain) of each morphism of $G$.
By abuse of notation, the set of \emph{composable pairs} of $G$
is denoted by 
$G^2 := \{(g,h) \in G \times G : s(g)=r(h)\}$.
For each $e \in G_0,$ we denote the corresponding \emph{isotropy group} by $G_e^e:=\{g \in G: s(g)=r(g)=e\}.$ 

\begin{definition}
Let $G$ be a groupoid.
\begin{enumerate}[{\rm (a)}]
    \item $G$ is said to be \emph{connected}, if for all $e,f\in G_0$, there exists $g\in G$ such that $s(g)=e$ and $r(g)=f$.
    \item A nonempty subset $H$ of $G$ is said to be a \emph{subgroupoid of $G$}, if
    $H^{-1} \subseteq H$
    and
    $gh\in H$ whenever $g,h\in H$ and $(g,h)\in G^2$.
\end{enumerate}
\end{definition}

\subsection{Groupoid graded rings}

\begin{definition}
Let $G$ be a groupoid.
A ring $S$ is said to be \emph{$G$-graded} (or \emph{graded by $G$})
if there is a collection $\{S_g\}_{g\in G}$ of additive subgroups of $S$  
such that
$S=\oplus_{g\in G} S_g$,
and
$S_g S_h \subseteq S_{gh}$, if $(g,h) \in G^2$, and $S_g S_h = \{0\}$, otherwise.
\end{definition}

\begin{obs}\label{P2-rem:GradedRings}
Suppose that $G$ is a groupoid and that $S$ is a $G$-graded ring.

(a)
If $H$ is a subgroupoid of $G$, then $S_H:=\oplus_{h\in H} S_h$ is an $H$-graded subring of $S$.
In particular, note that $\oplus_{e\in G_0} S_e$ and $\oplus_{g\in G_e^e} S_g$, for every $e\in G_0$, are subrings of $S$.

(b)
For any element $c = \sum_{g\in G} c_g \in S$, with $c_g \in S_g$, we define
$\Supp(c):= \{g\in G : c_g \neq 0\}$.

(c) An ideal $I$ of $S$ is said to be a \emph{graded ideal} (or \emph{$G$-graded ideal}) if
$I=\oplus_{g \in G} (I \cap S_g)$.
\end{obs}

The next lemma generalizes \cite[Lem.~2.4]{Oinert2019}.
For the convenience of the reader, we include a proof.

\begin{lema}\label{P2-lemmapi}
Let $G$ be a groupoid and let $S$ be a $G$-graded ring.
Suppose that $H$ is a subgroupoid of $G$.
Define $\pi_H : S \to S_H$ by
$\pi_H \left( \sum_{g\in G} c_g \right) := \sum_{h\in H} c_h.$
The following assertions hold:
\begin{enumerate}[{\rm (i)}]
\item The map $\pi_H : S \to S_H$ is additive.
\item If $a\in S$ and $b\in S_H$, then $\pi_H(ab)=\pi_H(a)b$ and $\pi_H(ba)=b\pi_H(a)$. 
\end{enumerate}
\end{lema}

\begin{proof}
(i) This is clear.

(ii)
Take $a\in S$ and $b\in S_H$.
Put $a':=a-\pi_H(a)$.
Clearly, $a = a' + \pi_H(a)$ and $\Supp(a')\subseteq G \setminus H$.
If $g\in G \setminus H$ and $h\in H$, then either the composition $gh$ does not exist or it belongs to $G \setminus H$. 
Thus, $\Supp(a'b) \subseteq \emptyset \cup G \setminus H$.
Hence,
$\pi_H(ab)=\pi_H((a'+\pi_H(a))b) = \pi_H(a'b) + \pi_H(\pi_H(a)b) = 0 + \pi_H(a)b.$ 
Analogously, one may show that 
$\pi_H(ba)=b\pi_H(a)$.
\end{proof}

\subsection{$s$-unital rings}

We briefly recall the definitions  of $s$-unital modules and rings
as well as some key properties.

\begin{definition}[{\cite[cf.~Def.~4]{Nystedt2019}}]
Let $R$ be an associative ring and let $M$ be a left (resp. right) $R$-module. We say that $M$ is \emph{$s$-unital} if 
$m \in R m$ (resp. $m \in m R$) 
for every $m \in M$. If $M$ is an $R$-bimodule, then we say that $M$ is \emph{$s$-unital} if it is $s$-unital both as a
left $R$-module and as a right $R$-module. The ring $R$ is said to be \emph{left $s$-unital} (resp. \emph{right $s$-unital}) if it is left (resp. right) $s$-unital as a left (resp. right) module over itself. The
ring $R$ is said to be \emph{$s$-unital} if it is $s$-unital as a bimodule over itself.
\end{definition} 

The following results are due to Tominaga \cite{Tominaga1976}.
For the proofs, we refer the reader to 
\cite[Prop.~2.8, Prop.~2.10]{Nystedt2019}.

\begin{proposition}
\label{P2:prop-s-unital}
Let $R$ be a ring and let $M$ be a left (resp. right) $R$-module. Then $M$ is left
(resp. right) $s$-unital if, and only if, for all $n \in \mathbb{Z}^+$ and all $m_1, \ldots ,m_n \in M$ there
is some $a \in R$ such that 
$a m_i = m_i$ (resp. $m_i a = m_i$)
for every $i \in \{1, \ldots, n\}$. 
\end{proposition}

\begin{proposition}\label{P2:prop-s-unitalBiModule}
Let $R$ be a ring and let $M$ be an $R$-bimodule. Then $M$ is $s$-unital if, and only if, for all $n \in \mathbb{Z}^+$ and all $m_1, \ldots ,m_n \in M$ there
is some $a \in R$ such that 
$a m_i =m_i a = m_i$
for every $i \in \{1, \ldots, n\}$. 
\end{proposition}

\begin{obs}
The element $a$, in Proposition~\ref{P2:prop-s-unitalBiModule}, is commonly  referred to as an \emph{$s$-unit for the set $\{m_1,\ldots,m_n\}$}.
\end{obs}

\section{Groupoid graded rings}\label{P2-primeness}

Throughout this section, unless stated otherwise, $G$ denotes an arbitrary groupoid.

\subsection{Nearly epsilon-strongly groupoid graded rings}\label{P2-sec:BasicProp}

In this section, we will recall the notion of a nearly epsilon-strongly groupoid graded ring and record some of its basic properties.

\begin{definition}[{\cite[Def.~4.6]{Lannstrom2022}}]
Let $S$ be a $G$-graded ring.
We say that $S$ is
\emph{nearly epsilon-strongly $G$-graded} if, for each $g\in G$,  $S_g S_{g^{-1}}$ is an $s$-unital ring
and
$S_g S_{g^{-1}} S_g = S_g$.
\end{definition}

\begin{obs}
The above definition simultaneously generalizes \cite[Def.~3.3]{Nystedt2020} and \cite[Def.~34]{NOP2020}.
\end{obs}

The following characterization of a nearly epsilon-strongly groupoid graded ring appeared in \cite[Prop.~15]{Lannstrom2022} without a proof.
For the convenience of the reader, we provide it here.

\begin{proposition}\label{P2-equiv-nes}
Let $S$ be a $G$-graded ring.
The following statements are equivalent:
\begin{enumerate}[{\rm(i)}]
\item $S$ is nearly epsilon-strongly $G$-graded;
\item For all $g \in G$ and $d \in S_g,$ there exist $\epsilon_g(d) \in S_g S_{g^{-1}}$ and $\epsilon'_g(d) \in S_{g^{-1}}S_g$ such that $\epsilon_g(d)d=d\epsilon'_g(d)=d.$
\end{enumerate}
\end{proposition}

\begin{proof} 
Suppose that (i) holds.
Let $g \in G$ and $d \in S_g.$
We may write
$d =\sum_{i=1}^n a_ib_ic_i$
for some $n\in \mathbb{Z}^+$,
$a_1,\ldots,a_n,c_1,\ldots,c_n \in S_g$ and $b_1,\ldots,b_n \in S_{g^{-1}}$.
Notice that $a_ib_i \in S_g  S_{g^{-1}}$ and $b_i c_i \in  S_{g^{-1}}S_g$ for every $i \in \{1,\ldots,n\}.$ By assumption,  $S_g  S_{g^{-1}}$ and $S_{g^{-1}}S_g$ are $s$-unital and hence,
by Proposition~\ref{P2:prop-s-unital},
there exist $\epsilon_g(d) \in S_g  S_{g^{-1}}$ and $\epsilon'_g(d) \in S_{g^{-1}}S_g$ such that $\epsilon_g(d)a_ib_i=a_ib_i$ and $b_i c_i \epsilon'_g(d)=b_i c_i$ for every $i \in \{1,\ldots,n\}.$ Thus, $\epsilon_g(d)d=d\epsilon'_g(d)=d.$
This shows that (ii) holds.

Conversely, suppose that (ii) holds.
Let $g \in G.$ 
Note that, by assumption, $S_g$ is $s$-unital as a left $S_gS_{g^{-1}}$-module and $S_{g^{-1}}$ is $s$-unital as a right $S_g S_{g^{-1}}$-module.
Let $m \in S_g S_{g^{-1}}.$ 
We may write 
$m=\sum_{i=1}^n a_i b_i$
for some $n \in \mathbb{Z}^+,$ $a_1,\ldots,a_n \in S_g$ and $b_1,\ldots,b_n \in S_{g^{-1}}$.
By $s$-unitality of the left $S_g S_{g^{-1}}$-module $S_g$, and Proposition~\ref{P2:prop-s-unital}, there is some
$u \in S_gS_{g^{-1}}$ such that $u a_i=a_i$ for every $i \in \{1,\ldots,n\}.$
Similarly, there is some
$u' \in S_gS_{g^{-1}}$ such that $b_i u'=b_i$ for every $i \in \{1,\ldots,n\}.$
Hence, $u m = m$ and $m u'=m$.
This shows that $S_g S_{g^{-1}}$ is $s$-unital.
Note that $S_g S_{g^{-1}}S_g \subseteq S_{gg^{-1}g} = S_g.$ 
Using that $S_g$ is $s$-unital as a left $S_g S_{g^{-1}}$-module we get that
$S_g \subseteq (S_g S_{g^{-1}}) S_g$.
Thus, $S_g=S_g S_{g^{-1}} S_g$.
This shows that (i) holds.
\end{proof}

The following result generalizes \cite[Prop.~2.13]{Lannstrom2021} from the group setting.

\begin{proposition}\label{P2-prop:NearlyProperties}
Let $S$ be a nearly epsilon-strongly $G$-graded ring. The following assertions hold:
\begin{enumerate}[{\rm (i)}]
\item $S_e$ is an $s$-unital ring, for every $e\in G_0$.

\item $d \in d (\oplus_{e \in G_0}S_e) \cap (\oplus_{e \in G_0}S_e)d$, for every $d \in S.$ 

\item $S$ is $s$-unital and $\oplus_{e \in G_0}S_e$ is an $s$-unital subring of $S.$

\item Suppose that $H$ is a subgroupoid of $G.$
Then $\oplus_{h \in H}S_h$ is a nearly epsilon-strongly $H$-graded ring. 

\item $\oplus_{g \in G_e^e}S_g$ is an $s$-unital ring, for every $e \in G_0.$

\item The set
$$G':=\{g \in G: S_{s(g)}\neq\{0\} \text{ and } S_{r(g)}\neq\{0\}\}$$
is a subgroupoid of $G$. 

\item $S=\oplus_{g\in G'}S_g.$

\end{enumerate}
\end{proposition}

\begin{proof}
(i) 
Take $e\in G_0$.
By assumption, 
 $S_e=S_eS_eS_e$ and $S_eS_e$ is an $s$-unital ring. Hence, 
 $S_e=(S_eS_e)S_e \subseteq S_eS_e  \subseteq S_{e}.$
 Thus, $S_e=S_eS_e$ is an $s$-unital ring.

(ii)
Let $d=\sum_{t \in G}d_t \in S$, with $d_t \in S_t.$ Take $g \in \Supp(d)$. 
By Proposition~\ref{P2-equiv-nes}, there exist $u_g \in S_g S_{g^{-1}}\subseteq S_{r(g)}$ and $u_{g^{-1}}' \in S_{g^{-1}}S_g \subseteq S_{s(g)}$ such that $u_gd_g=d_gu_{g^{-1}}'=d_g.$
The set $B:=\{r(t): t \in \Supp(d) \}$ is finite, because
 $\Supp(d)$ is finite.
 For every $f\in B$, 
 by 
 (i),
 $S_f$ is $s$-unital,
 and we let $v_f \in S_f$ be an $s$-unit for the finite set $\{u_t: t \in \Supp(d) \text{ and } r(t)=f \} \subseteq S_f.$ Define $a:=\sum_{f \in B}v_f \in \oplus_{e \in G_0}S_e.$
We get that 
\begin{align*}
a d &=
\sum\limits_{f \in B}v_f \sum\limits_{t \in G}d_t
=\sum\limits_{t \in \Supp(d)} v_{r(t)} d_t
=\sum\limits_{t \in \Supp(d)}v_{r(t)}(u_td_t) \\
&=\sum\limits_{t \in \Supp(d)}(v_{r(t)}u_t)d_t=\sum\limits_{t \in \Supp(d)}u_td_t=\sum\limits_{t \in G}d_t=d.
\end{align*}

Similarly, we define the finite set $D:=\{s(t) : t \in \Supp(d) \}.$ For every $f \in D,$ we let $v_{f}' \in S_{f}$ be an $s$-unit for the finite set $\{u_{t^{-1}}' : t \in \Supp(d) \text{ and } s(t)=f \} \subseteq S_{f}.$ Define $a':=\sum\limits_{f \in D}v_{f}' \in \oplus_{e \in G_0}S_e.$ We get that
\begin{align*}
d a' &=\sum\limits_{t \in G}d_t\sum\limits_{f \in D}v_{f}'
=\sum\limits_{t \in \Supp(d)} d_t v_{s(t)}' 
=\sum\limits_{t \in \Supp(d)}(d_t u_{t^{-1}}') v_{s(t)}' \\
&=\sum\limits_{t \in \Supp(d)}d_t(u_{t^{-1}}' v_{s(t)}')=\sum\limits_{t \in \Supp(d)}d_t u_{t^{-1}}'=\sum\limits_{t \in G}d_t=d.
\end{align*}

(iii)
It follows immediately from (ii).

(iv)
It follows immediately from Remark~\ref{P2-rem:GradedRings}
and the fact that $H\subseteq G$.

(v)
Take $e \in G_0$.
Clearly, the isotropy group $G_e^e$ is a subgroupoid of $G$. By
(iv) and (iii) we get that $\oplus_{g \in G_e^e}S_g$ is $s$-unital.

(vi)
Clearly, $g^{-1} \in G'$ whenever $g\in G'$.
Suppose that $g,h \in G'$ and $(g,h)\in G^2$. Then $S_{s(gh)}=S_{s(h)}\neq\{0\}$ and $S_{r(gh)}=S_{r(g)}\neq\{0\}.$ Therefore, $gh \in G'.$ 
This shows that $G'$ is a subgroupoid of $G$.

(vii)
Take $g \in G$ such that $S_g \neq \{0\}$. We claim that $g\in G'$. If we assume that the claim holds, then clearly $S=\oplus_{g\in G'}S_g.$
Now we show the claim.
Let $d \in S_g$ be  nonzero. By Proposition~\ref{P2-equiv-nes}, there are $\epsilon_g(d) \in S_gS_{g^{-1}} \subseteq S_{r(g)}$ and $\epsilon_g'(d) \in S_{g^{-1}}S_g \subseteq S_{s(g)}$ such that $\epsilon_g(d)d=d\epsilon_g'(d)=d.$
In particular, $S_{r(g)} \neq \{0\}$ and $S_{s(g)} \neq \{0\}.$
\end{proof}

\begin{obs}
(a)
Note that (vi) above holds for any $G$-graded ring. For (vii), however, the nearly epsilon-strongness of the $G$-grading is necessary.

(b)
Suppose that $S$ is a nearly epsilon-strongly $G$-graded ring.
By Proposition~\ref{P2-prop:NearlyProperties} (vii),
$g\in G'$ whenever $S_g \neq \{0\}$. The converse, however, need not hold (see Example~\ref{P2:ex-partial-connected-not-prime}).
\end{obs}

\begin{example}\label{P2-exgraded}
Let $G:=\{f_1, f_2, f_3, g, h, g^{-1}, h^{-1}, hg^{-1},gh^{-1}\}$ be a groupoid with $G_0=\{f_1, f_2, f_3\}$ and depicted as follows:

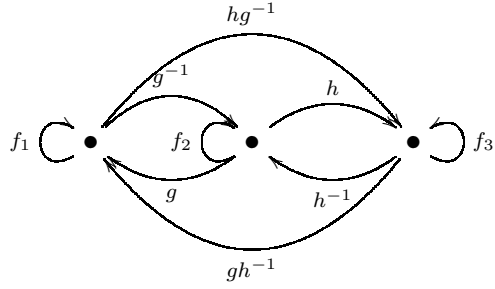
\[
\xymatrixcolsep{4pc}
\xymatrix{{\bullet} \ar@(dl,ul)^{f_1} \ar@/^3.4pc/[rr]^{hg^{-1}} \ar@/^1.45pc/[r]^{g^{-1}} & {\bullet} 
\ar@(dl,ul)^{f_2} 
\ar@/^1.2pc/[l]^g \ar@/^1.2pc/[r]^h &
{\bullet} \ar@(dr,ur)_{f_3} \ar@/^1.2pc/[l]^{h^{-1}} \ar@/^3.4pc/[ll]^{gh^{-1}}} 
\]

Let $S:=\mathcal{M}_{3}(\mathbb{Z}),$ be the ring of  $3\times 3$ matrices over $\mathbb{Z},$
and let $\{e_{ij}\}_{i,j},$ denote the standard matrix units.
We define: 
$$\begin{aligned}
    S_g &:=\mathbb{Z} e_{12}, \hspace{1cm}  & S_{g^{-1}}:=\mathbb{Z} e_{21},  \hspace{1cm} & S_h:= \mathbb{Z} e_{32}, \\
     S_{h^{-1}} & := \mathbb{Z} e_{23}, \hspace{1cm} & S_{gh^{-1}}:=\mathbb{Z} e_{13}, \hspace{1cm} & S_{hg^{-1}}:=\mathbb{Z} e_{31}, \\
      S_{f_1} & :=\mathbb{Z} e_{11}, \hspace{1cm} & S_{f_2}:= \mathbb{Z} e_{22}, \hspace{1cm} & S_{f_3}:=\mathbb{Z} e_{33}.
\end{aligned}$$

Notice that $S=\oplus_{l \in G}S_l.$ 
It is not difficult to see that
the ring $\mathcal{M}_{3}(\mathbb{Z})$ is nearly epsilon-strongly $G$-graded. 
\end{example}

\subsection{Invariance in groupoid graded rings}

Inspired by \cite[Sec.~3--4]{Lannstrom2021}, 
we shall now examine the relationship between $G$-graded ideals of a $G$-graded ring $S$ and $G$-invariant ideals of the subring $\oplus_{e \in G_0}S_e.$ 
Throughout this section, $S$ denotes an arbitrary $G$-graded ring.

\begin{definition}[{\cite[Def.~3.1, Def.~3.3]{Lannstrom2021}}]\label{P2-def-invariant-nearly}
Let $S$ be a $G$-graded ring.
\begin{enumerate}[{\rm(i)}]

\item  For any $g\in G$ and any subset $I$ of $S$, we write $I^g := S_{g^{-1}} I S_g.$

\item Let $H$ be a subgroupoid of $G$ and let $I$ be a subset of $S.$ Then, $I$ is called \emph{$H$-invariant} if $I^g \subseteq I$ for every $g \in H.$

\end{enumerate}
\end{definition}

\begin{obs}
Note that if $g \in G,$ and $I \subseteq S,$ then 
$$I^g =\left\{\sum\limits_{k=1}^n a_kx_kb_k : n \in \mathbb{Z}^+, a_k \in S_{g^{-1}}, x_k \in I \text{ and } b_k \in S_g \text{ for each } k \in \{1,\ldots, n\} \right\}.$$
\end{obs}

\begin{lema} 
If $g \in G$ and $J$ is an ideal of $\oplus_{e \in G_0}S_e,$ then $J^g$ is an ideal of $\oplus_{e \in G_0}S_e.$
\end{lema}

\begin{proof} 
Let $g\in G$ and let $J$ be an ideal of $\oplus_{e \in G_0}S_e$. Notice that $J^g$ is an additive subgroup of $\oplus_{e \in G_0}S_e.$ Moreover,
$(\oplus_{e \in G_0}S_e)J^g=(\oplus_{e \in G_0}S_e)(S_{g^{-1}}JS_g) =  S_{s(g)}S_{g^{-1}}JS_g \subseteq S_{g^{-1}}JS_g = J^g.$ Analogously,
$J^g(\oplus_{e \in G_0}S_e) \subseteq J^g.$ 
\end{proof}

\begin{proposition}\label{P2:prop-SJS-graded} 
Suppose that $J$ is an ideal of $\oplus_{e \in G_0}S_e.$ Then $SJS$ is a $G$-graded ideal of $S.$
\end{proposition}

\begin{proof}
It is clear that $SJS$ is an ideal of $S$ and that 
$\oplus_{g \in G}((SJS)\cap S_g) \subseteq SJS.$ Now, we show the reversed inclusion. 
Take $g,h\in G$, $a_g \in S_g,$ $c_h \in S_h$ and  $b=\sum_{e \in G_0}b_e \in J.$ 
If $s(g)\neq r(h),$ then $a_g b c_h=0 \in (SJS)\cap S_{gh}.$
Otherwise, $s(g)=r(h)$, and then $a_g b c_h= a_gb_{s(g)}c_h \in (SJS)\cap S_{gh}.$ Thus, $SJS=\oplus_{g \in G}((SJS)\cap S_g).$   
\end{proof}

\begin{lema}\label{P2-lemma:equal-inv}
Suppose that $\oplus_{e \in G_0}S_e$ is $s$-unital and that $J$ is an ideal of $\oplus_{e\in G_0}S_e.$ Then $J$ is $G$-invariant if, and only if, $(SJS)\cap\oplus_{e \in G_0}S_e=J.$ 
\end{lema}

\begin{proof}
We first show the ``only if'' statement.
Suppose that $J$ is $G$-invariant. For each $e \in G_0,$ we have $$(SJS)\cap S_e \subseteq \left(\sum_{\substack{g \in G \\ s(g)=e}} (S_{g^{-1}}JS_g)\right)\cap S_e \subseteq \left(\sum_{\substack{g \in G \\ s(g)=e}}J^g\right)\cap S_e \subseteq J.$$ 
Let $a = \sum_{f\in G_0} a_f \in (SJS) \cap \oplus_{e\in G_0} S_e$.
By Proposition~\ref{P2:prop-SJS-graded}, $SJS$ is $G$-graded and we notice that 
$a_f \in (SJS) \cap  S_f \subseteq J$.
Thus, 
$(SJS) \cap \oplus_{e\in G_0}S_e \subseteq J$.
By assumption, $\oplus_{e\in G_0}S_e$ is $s$-unital and $J \subseteq \oplus_{e \in G_0}S_e$. 
 Hence, $J \subseteq (\oplus_{e \in G_0}S_e)J(\oplus_{e \in G_0}S_e) \subseteq (SJS)\cap\oplus_{e \in G_0}S_e.$

Now we show the ``if'' statement.
Suppose that $(SJS)\cap\oplus_{e \in G_0}S_e=J.$ Take $g \in G$ and notice that $J^g=S_{g^{-1}}JS_g \subseteq (SJS)\cap S_{s(g)}\subseteq (SJS)\cap\oplus_{e \in G_0}S_e=J.$ Thus, $J$ is 
$G$-invariant.
\end{proof}

\begin{lema}\label{P2:lemma-inv-ideal} 
If $I$ is a $G$-graded ideal of $S,$ then $I \cap \oplus_{e \in G_0}S_e$ is a $G$-invariant ideal of $\oplus_{e \in G_0}S_e.$
\end{lema}

\begin{proof}
Let $I$ be a $G$-graded ideal of $S$.
Clearly, $I \cap \oplus_{e \in G_0}S_e$ is an ideal of $\oplus_{e \in G_0}S_e.$
Take $g \in G.$ Notice that $S_{g^{-1}}(I \cap S_{r(g)})S_g \subseteq I \cap S_{s(g)}.$ Furthermore, if $e \in G_0\setminus \{r(g)\},$ then $S_{g^{-1}}(I \cap S_e)S_g =\{0\}.$ Therefore, $(I \cap \oplus_{e \in G_0}S_e)^g = S_{g^{-1}}(I \cap \oplus_{e \in G_0}S_e)S_g \subseteq I \cap \oplus_{e \in G_0}S_e.$ 
\end{proof}

\begin{lema}\label{P2-lemma:equal-graded}
Let $S$ be a nearly epsilon-strongly $G$-graded ring. If $I$ is a $G$-graded ideal of $S,$ then $I=(I\cap \oplus_{e \in G_0}S_e)S=S(I\cap \oplus_{e \in G_0}S_e)=S(I\cap \oplus_{e \in G_0}S_e)S.$  
\end{lema}

\begin{proof}
Let $I$ be a $G$-graded ideal of $S.$ 
By Proposition~\ref{P2-prop:NearlyProperties}, $S$ is $s$-unital and hence $(I\cap \oplus_{e \in G_0}S_e) \subseteq (I\cap \oplus_{e \in G_0}S_e)S.$ Thus, $S(I\cap \oplus_{e \in G_0}S_e) \subseteq S(I\cap \oplus_{e \in G_0}S_e)S \subseteq I.$ Analogously, $(I\cap \oplus_{e \in G_0}S_e)S \subseteq S(I\cap \oplus_{e \in G_0}S_e)S \subseteq I.$ 

We claim that $I \subseteq S(I\cap \oplus_{e \in G_0}S_e).$ Take $g\in G$ and $a_g \in I \cap S_g.$ By Proposition~\ref{P2-equiv-nes}, there is some $u_g \in S_gS_{g{-1}}$ such that $a_g=u_ga_g.$ Then, $u_g=\sum_{j=1}^n b_jc_j$ for some $n\in \mathbb{Z}^+$, $b_1,\ldots,b_n\in S_g$ and $c_1,\ldots,c_n \in S_{g^{-1}}.$
Notice that $c_ja_g \in (S_{g^{-1}}S_g)\cap I\subseteq S_{s(g)}\cap I$ for every $j \in \{1,\ldots,n\}.$ Hence, $a_g=u_ga_g=\sum_{j=1}^n b_jc_ja_g \in S_g(S_{s(g)}\cap I)\subseteq S(\oplus_{e \in G_0}S_e\cap I).$ 
Using that $I$ is $G$-graded,
we get that
$I \subseteq S(\oplus_{e \in G_0}S_e\cap I).$
Similarly, $I \subseteq (\oplus_{e \in G_0}S_e\cap I)S.$ Thus, $I=S(\oplus_{e \in G_0}S_e\cap I)=(\oplus_{e \in G_0}S_e\cap I)S$ and $I=S(\oplus_{e \in G_0}S_e\cap I)\subseteq S(\oplus_{e \in G_0}S_e\cap I)S \subseteq I.$
\end{proof}

\begin{corolario}\label{P2-cor:inv-SJ-JS}
Let $S$ be a nearly epsilon-strongly $G$-graded ring. If $J$ is a $G$-invariant ideal of $\oplus_{e \in G_0}S_e,$ then $SJS=JS=SJ.$
\end{corolario}

\begin{proof}
Let $J$ be a $G$-invariant ideal of $\oplus_{e \in G_0}S_e.$ 
By Proposition~\ref{P2:prop-SJS-graded},  $SJS$ is a $G$-graded ideal of $S,$ and, by Lemma~\ref{P2-lemma:equal-graded},
$SJS=((SJS)\cap \oplus_{e \in G_0}S_e)S=S\cap ((SJS)\cap \oplus_{e \in G_0}S_e).$ 
Thus, by
Proposition~\ref{P2-prop:NearlyProperties} (iii) and Lemma~\ref{P2-lemma:equal-inv}, $SJS=JS=SJ.$
\end{proof}

By Lemmas~\ref{P2:lemma-inv-ideal} and \ref{P2:prop-SJS-graded}, the following maps are well defined: 
$$\phi: \{G\text{-graded ideals of } S\} \ni  I \mapsto I \cap \oplus_{e \in G_0}S_e \in  \{G\text{-invariant ideals of} \oplus_{e \in G_0}S_e \}$$
$$\psi: \{G\text{-invariant ideals of} \oplus_{e \in G_0}S_e\} \ni J \mapsto SJS \in  \{G\text{-graded ideals of } S\}.$$

The following theorem generalizes \cite[Thm.~4.7]{Lannstrom2021} and \cite[Thm.~3.11]{Moreira2022}. 

\begin{theorem}\label{P2-thm:InvGradedIdealCorresp}
Let $S$ be a nearly epsilon-strongly $G$-graded ring. The map $\phi$ defines a bijection between the set of $G$-graded ideals of $S$ and the set of $G$-invariant ideals of $\oplus_{e \in G_0}S_e.$ The inverse of $\phi$ is given by $\psi.$ 
\end{theorem}

\begin{proof} 
Let $I$ be a $G$-graded ideal of $S.$ Lemma~\ref{P2-lemma:equal-graded} implies that $\psi\circ \phi(I)=S(I \cap \oplus_{e \in G_0}S_e)S=I.$ 
Let $J$ be a $G$-invariant ideal of $S.$ Notice that, by Proposition~\ref{P2-prop:NearlyProperties}~(iii) and Lemma~\ref{P2-lemma:equal-inv}, $\phi\circ \psi(J)=(SJS)\cap \oplus_{e \in G_0}S_e=J.$ 
\end{proof}

\subsection{Graded primeness of groupoid graded rings}\label{graded-prime}

In this section, we identify necessary and sufficient conditions for graded primeness of a groupoid graded ring. 

\begin{definition}\label{P2-def-prime-nearly} 
Let
$S$ be a $G$-graded ring.
\begin{enumerate}[{\rm(i)}]

\item  
$\oplus_{e \in G_0}S_e$ is said to be \emph{$G$-prime} if there are no nonzero $G$-invariant ideals $I,J$ of $\oplus_{e \in G_0}S_e$ such that $IJ=\{0\}$. 

\item  
$S$ is said to be \emph{graded prime} if there are no nonzero $G$-graded ideals $I,J$ of $S$ such that $IJ=\{0\}$.
\end{enumerate}
\end{definition}

The following result generalizes \cite[Prop.~3.28]{Moreira2022}.

\begin{theorem}\label{P2-teo:graded-G-prime}
Let $S$ be a nearly epsilon-strongly $G$-graded ring. Then $S$ is graded prime if, and only if, $\oplus_{e \in G_0}S_e$ is $G$-prime.
\end{theorem}

\begin{proof} 
We first show the ``if'' statement.
Suppose that $\oplus_{e \in G_0}S_e$ is $G$-prime and let $I_1, I_2$ be nonzero $G$-graded ideals of $S.$ By Lemma~\ref{P2:lemma-inv-ideal} and Proposition~\ref{P2-equiv-nes}, $(I_1 \cap \oplus_{e \in G_0}S_e)$ and $(I_2 \cap \oplus_{e \in G_0}S_e)$ are nonzero $G$-invariant ideals of $\oplus_{e \in G_0}S_e.$ Then $\{0\} \neq (I_1 \cap \oplus_{e \in G_0}S_e)\cdot (I_2 \cap \oplus_{e \in G_0}S_e)\subseteq I_1 I_2.$

Now, we show the ``only if'' statement.
Suppose that $S$ is graded prime and let $J_1, J_2$ be nonzero $G$-invariant ideals of $\oplus_{e \in G_0}S_e.$ By
Proposition~\ref{P2-prop:NearlyProperties} (iii) and
Proposition~\ref{P2:prop-SJS-graded},  $SJ_1S$ and $SJ_2S$ are nonzero $G$-graded ideals of $S.$ By Corollary~\ref{P2-cor:inv-SJ-JS} and our assumption,
$SJ_1 \cdot J_2S = SJ_1S \cdot SJ_2S \neq \{0\}.$ Thus, $J_1\cdot J_2 \neq \{0\}$.
\end{proof}

Now, we determine some necessary conditions for graded primeness of a groupoid graded ring.

\begin{lema}\label{P2-lemmagraded1} Let $S$ be a $G$-graded ring which is $s$-unital. Suppose that $S$ is graded prime. 
Let $a_g \in S_g$ and $c_h \in S_h$ be nonzero elements, for some $g,h\in G$.
Then there is some $t\in G$ and $b_t \in S_t$ such that
$a_g b_t c_h$ is nonzero.
\end{lema}

\begin{proof}
We prove the contrapositive statement. Suppose that $a_g b_t c_h = 0$ for all $t\in G$ and  $b_t \in S_t.$ Consider the sets $A:=Sa_gS$ and $C:=Sc_hS.$ By $s$-unitality of $S$ it is clear that both $A$ and $C$ are $G$-graded nonzero ideals of $S.$ Moreover, by assumption we have $AC = Sa_gS Sc_hS = Sa_g Sc_hS = \{0\}.$ This shows that $S$ is not graded prime.
\end{proof}

\begin{definition}\label{P2:def-support-hub}
Let $S$ be a $G$-graded ring.
An element $e\in G_0'$ (see Proposition~\ref{P2-prop:NearlyProperties}~(vi)) is said to be a \emph{support-hub} if
for every nonzero $a_g \in S_g$, with $g \in G$, there are $h,k \in G$ such that $s(h)=e,$ $r(k)=e$, and $a_gS_h$  and $S_ka_g$ are both nonzero.
\end{definition}

\begin{obs}
Let $S$ be a $G$-graded ring.
\begin{enumerate}[{\rm(a)}]
\item  Suppose that $e \in G_0'$ is a support-hub and that $a_g \in S_g$ is nonzero, for some $g\in G$.
Notice that there are $h,k \in G$ as in the following diagram. 

\[
\xymatrix{
& e \ar[dl]_h  \\
s(g) \ar[r]_g & r(g) \ar[u]_k 
}    
\]

\item Notice that,
if $S$ is a ring which is nearly epsilon-strongly graded by a group $G$, then the identity element $e$ of $G$ is always a support-hub.
\end{enumerate}
\end{obs}

\begin{proposition}\label{P2-nearnec}Let $S$ be a $G$-graded ring which is $s$-unital.  If $S$ is graded prime, then every $e \in G_0'$ is a support-hub. 
\end{proposition}

\begin{proof}
We prove the contrapositive statement.
Suppose that there is some $e \in G_0'$ which is not a support-hub. Then there are $g\in G$ and a nonzero element $a_g \in S_g$, such that for every $h\in G$ such that $s(h)=e,$ we have that $a_gS_h=\{0\}$ or for every $k\in G$ such that $r(k)=e,$ we have that $S_ka_g=\{0\}.$ Let $a_e$ be a nonzero element of $S_e.$ Using that $S$ is $s$-unital, $A=Sa_eS$ and $B=Sa_gS$ are nonzero $G$-graded ideals of $S.$

Notice that if for every $h\in G$ such that $s(h)=e,$ we have that $a_gS_h=\{0\},$ then $BA=Sa_gSSa_eS=Sa_gSa_eS=\{0\}.$ Moreover, if for every $k\in G$ such that $r(k)=e,$ we have that $S_ka_g=\{0\},$ then $AB=Sa_eSSa_gS=Sa_eSa_gS=\{0\}.$ Therefore, $S$ is not graded prime.
\end{proof}

\begin{proposition}\label{P2:prop-connected}
Let $S$ be a $G$-graded ring which is $s$-unital. The following assertions hold:
\begin{enumerate} [{\rm(i)}]
    \item If $G$ is a connected groupoid, then $G'$ is a connected subgroupoid of $G$.
    \item 
    If there is a support-hub in $G_0'$, then $G'$ is a connected subgroupoid of $G$.
    
    \item If $S$ is graded prime, then $G'$ is a connected subgroupoid of $G$.
\end{enumerate}
\end{proposition}

\begin{proof}
(i) 
Suppose that $G$ is connected.
Take $e,f \in G_0'.$ By assumption, there is $g \in G$ such that $s(g)=e$ and $r(g)=f.$ Since $S_e$ and $S_f$ are nonzero, we must have $g \in G'$, and hence $G'$ is connected.

(ii) 
Suppose that $e\in G_0'$ is a support-hub.
Take $f_1,f_2 \in G_0'.$ By the definition of $G',$ there are nonzero elements $a_{f_1} \in S_{f_1}$ and $b_{f_2} \in S_{f_2}.$ Since $e$ is a support-hub, there is some $k_1 \in G$ such that $r(k)=e$ and $S_ka_{f_1}\neq \{0\}.$ In particular, $s(k)=f_1.$ Using again that $e$ is a support-hub, there is some $h \in G$ such that $s(h)=e$ and $b_{f_2}S_h \neq \{0\}.$ Then, $r(h)=f_2.$
Define $t:=hk$ and note that $r(t)=f_2$ and $s(t)=f_1.$

(iii) It follows from Proposition~\ref{P2-nearnec} and (ii).
\end{proof}

\subsection{Primeness of groupoid graded rings}

In this section, we will provide necessary and sufficient conditions for  primeness of a nearly epsilon-strongly $G$-graded ring.
Furthermore, we will extend \cite[Thm.~1.3]{Lannstrom2021} to the context of groupoid graded rings. 

\begin{proposition}\label{P2-iso} 
Let $S$ be a nearly epsilon-strongly $G$-graded ring. If $S$ is prime, then $\oplus_{g \in G_e^e}S_g$ is prime for every $e \in G_0.$
\end{proposition}

\begin{proof} 
We prove the contrapositive statement. Let $e \in G_0.$ Suppose that $I,J$ are nonzero ideals of $\oplus_{g \in G_e^e}S_g$ such that $IJ=\{0\}$. 
By Proposition~\ref{P2-prop:NearlyProperties}, $S$ is $s$-unital and hence $A':= SIS$ and $B':=SJS$ are nonzero ideals of $S.$ 
Clearly,  $A'B'=SISSJS=SISJS.$ We claim that $ISJ \subseteq IJ.$ 
If we assume that the claim holds, then it follows that $A'B'=\{0\}$, and we are done. 
Now we show the claim.
Take $g\in G$, $c_g \in S_g,$
$a=\sum_{k \in G_e^e} a_k \in I,$ 
and $b=\sum_{t\in G_e^e} b_t \in J$.
Let $k, t \in G_e^e.$ If $e=s(k)\neq r(g)$ or $s(g) \neq r(t)=e$, then $a_kc_gb_t=0 \in IJ.$ Otherwise, $g \in G_e^e,$ and then, since $I$ and $J$ are ideals of $\oplus_{g \in G_e^e}S_g$, we get that $ac_gb \in IJ.$
Thus, $ISJ \subseteq IJ$. 
\end{proof}

\begin{obs} 
Let $S$ be a nearly epsilon-strongly $G$-graded ring.  By Proposition~\ref{P2-iso} and Proposition~\ref{P2:prop-connected} (iii),  if $S$ is prime, then $\oplus_{g \in G_e^e}S_g$ is prime for every $e \in G_0$ and $G'$ is connected. The converse, however,  
need not hold
as shown in Example~\ref{P2:ex-partial-connected-not-prime}.
\end{obs}

The next result partially generalizes \cite[Lem.~2.19]{Lannstrom2021}.

\begin{lema}\label{P2-lema-support-hub}
Let $S$ be a nearly epsilon-strongly $G$-graded ring and let $I$ be a nonzero ideal of $S$.
If $e \in G_0'$ is a support-hub, 
then  $\pi_{G_e^e}(I)$ is a nonzero ideal of $\oplus_{g \in G_e^e}S_g.$
\end{lema}

\begin{proof}
Suppose that $e\in G_0'$ is a support-hub.
By Lemma~\ref{P2-lemmapi}, $\pi_{G_e^e}(I)$ is an ideal of $\oplus_{g \in G_e^e}S_g.$
We claim that $\pi_{G_e^e}(I)\neq \{0\}.$ Let $d= d_{g_1} + d_{g_2} + \ldots + d_{g_n} \in I$ be an element where all the homogeneous coefficients are nonzero and the $g_i's$ are distinct. By Proposition~\ref{P2-equiv-nes}, there is some nonzero $c_{g_1^{-1}} \in S_{g_1^{-1}}$ such that $d_{g_1} c_{g_1^{-1}}$ is nonzero and contained in $S_{r(g_1)}$.

Notice that $d c_{g_1^{-1}}$ is nonzero and contained in $I$. Thus, without loss of generality, we may assume that $g_1 \in G_0'$. Since $e$ is a support-hub, there is an element $k \in G$ such that $r(k)=e$ and  $S_kd_{g_1}$ is nonzero. In particular, there is an element $b_k \in S_k$ such that 
$b_k d_{g_1}$ 
is nonzero. Using again that $e$ is a support-hub, there is an element $h \in G$ such that $s(h)=e$ and $(b_k d_{g_1}) S_h$ is nonzero. Therefore, there is an element $b_h \in S_h$ such that $(b_k d_{g_1})b_h$ is nonzero. Thus, $b_k d b_h \in I$ and $$\pi_{G_e^e}(b_k d b_h)= \sum\limits_{\substack{s(k)=r(g_i) \\ s(g_i)=r(h)}}b_k d_{g_i} b_h= \sum\limits_{\substack{g_1=r(g_i) \\ s(g_i)=g_1}}b_k d_{g_i} b_h.$$
Notice that 
$b_k d_{g_i} b_h \in S_{kh} \setminus \{0\}$ if, and only if, $g_i=g_1.$ Thus, $0 \neq \pi_{G_e^e}(b_k d b_h) \in \pi_{G_e^e}(I).$
\end{proof}

\begin{theorem}\label{P2-nearsuf}
Let $S$ be a nearly epsilon-strongly $G$-graded ring. If there is some $e \in G_0'$ such that $e$ is a support-hub 
and $\oplus_{g \in G_e^e}S_g$ is prime, 
then $S$ is prime.
\end{theorem}

\begin{proof}
Suppose that $e\in G_0'$ is a support-hub and that $\oplus_{g \in G_e^e}S_g$ is prime.
Let $I$ and $J$ be nonzero ideals of $S.$ By Lemma~\ref{P2-lema-support-hub}, $\pi_{G_e^e}(I)$ and $\pi_{G_e^e}(J)$ are nonzero ideals of $\oplus_{g \in G_e^e}S_g,$ and hence, by assumption, $\pi_{G_e^e}(I)\pi_{G_e^e}(J) \neq \{0\}.$

We claim that $\pi_{G_e^e}(I) \subseteq I.$ Let $d=\sum_{t \in G}d_t \in I$ and consider the finite set $F:=\{t \in \Supp(d): r(t)=e \}.$ 
Take $t\in F$.
Using that $S_t=S_tS_{t^{-1}} S_t,$ there are $n_t\in \mathbb{Z}^+$, $a_1^t,\ldots,a_{n_t}^t \in S_{t}S_{t^{-1}}\subseteq S_{r(t)}$ and $b_1^t, \ldots, b_{n_t}^t \in S_t$ such that
$d_t=\sum_{i=1}^{n_t} a_i^tb_i^t.$
Since $S_{r(t)}=S_e$ is $s$-unital, there is some $u_e \in S_{r(t)}$ such that $u_ea_i^t=a_i^t$ for all $i\in \{1,\ldots,n_t\}.$ Thus, $d_t=u_ed_t$.
Therefore,
$u_e\sum_{t \in G}d_t= \sum_{t\in F}u_e d_t=\sum_{t\in F}d_t.$

Using a similar argument, there is some $v_e \in S_e$ such that $d_t=d_t v_e$ for every $t\in \Supp(d)$ such that $s(t)=e$.
Therefore,
$$I \ni \left(u_e\sum\limits_{t \in G}d_t\right)v_e=\left(\sum\limits_{t\in F}d_t\right)v_e=\sum\limits_{t \in G_e^e}d_tv_e=\sum\limits_{t \in G_e^e}d_t = \pi_{G_e^e}(d).$$
Analogously, $\pi_{G_e^e}(J) \subseteq J.$ Thus, $\{0\} \neq \pi_{G_e^e}(I)\pi_{G_e^e}(J) \subseteq IJ$ and $S$ is prime.  
\end{proof}

\begin{obs}
The assumption on the existence of a support-hub in Theorem~\ref{P2-nearsuf} cannot be dropped.
Indeed, consider the groupoid $G=\{e,f\}=G_0$ and the groupoid ring $S:=\mathbb{C}[G]$.
Then $G_e^e = \{e\}$ and $G_f^f=\{f\}$.
Furthermore, $S_e \cong \mathbb{C}$ and $S_f \cong \mathbb{C}$ are both prime.
Nevertheless, $S$ is not prime.
\end{obs}

\begin{example}\label{P2-exnotgraded}
Let $G:=\{f_1, f_2, f_3, g, h, g^{-1}, h^{-1}, hg^{-1},gh^{-1}\}$ be a groupoid with $G_0=\{f_1, f_2, f_3\}$ and depicted as follows:
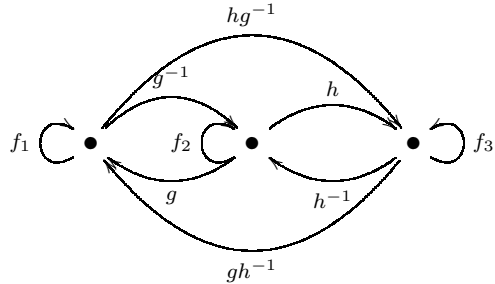
\[
\xymatrixcolsep{4pc}
\xymatrix{{\bullet} \ar@(dl,ul)^{f_1} \ar@/^3.4pc/[rr]^{hg^{-1}} \ar@/^1.45pc/[r]^{g^{-1}} & {\bullet} \ar@(dl,ul)^{f_2} \ar@/^1.2pc/[l]^g \ar@/^1.2pc/[r]^h &
{\bullet} \ar@(dr,ur)_{f_3} \ar@/^1.2pc/[l]^{h^{-1}} \ar@/^3.4pc/[ll]^{gh^{-1}}} 
\]
Define $S$ as the ring of matrices over $\mathbb{Z}$ of the form
$$\begin{pmatrix}
a_{11} & a_{12} & 0       & 0 \\
a_{21} & a_{22} & 0       & 0 \\
  0    &   0    & a_{33}  & a_{34} \\
  0    &   0    & a_{43}  & a_{44}
\end{pmatrix}.$$
Denote by $\{e_{ij}\}_{i,j},$ the canonical basis of $S$ and define: 
$$ S_g:=\mathbb{Z} e_{12}, \hspace{1cm} S_{g^{-1}}:=\mathbb{Z} e_{21},
\hspace{1cm} S_h:=\mathbb{Z} e_{43}, \hspace{1cm}
S_{h^{-1}}:= \mathbb{Z} e_{34},$$ $$S_{f_1}:= \mathbb{Z} e_{11},
\hspace{1cm}  S_{f_3}:= \mathbb{Z} e_{44}, \hspace{1cm} S_{f_2}:=\{ \lambda_1 e_{22} + \lambda_2 e_{33} : \lambda_1, \lambda_2 \in \mathbb{Z}\},$$ 
and $S_l:=\{0\},$ otherwise. Notice that $S=\oplus_{l \in G}S_l$ and that $f_2 \in G_0'$ is a support-hub.
However, observe that $e_{12} \in S_g$ and $e_{43}\in S_h$ are nonzero elements, and that there is no element $a_l \in S_l$ such that $l \in G'$ and $e_{12}a_le_{43}\neq 0.$ Therefore, by Lemma~\ref{P2-lemmagraded1}, $S$ is not graded prime.
\end{example}

Now, we prove our main result.

\begin{proof}[Proof of Theorem~\ref{P2-teo-nesgr}] 
It follows from Proposition~\ref{P2-iso} and by the definition of primeness that (i) $\Rightarrow$ (iv) $\Rightarrow$ (v). By Proposition~\ref{P2-prop:NearlyProperties}~(iii), $S$ is $s$-unital and Proposition~\ref{P2-nearnec} implies (v) $\Rightarrow$ (vii).
By Theorem~\ref{P2-nearsuf}, (vii) $\Rightarrow$ (i). 
Finally, note that by Theorem~\ref{P2-teo:graded-G-prime}, (ii) is equivalent to (iv), and (iii) is equivalent to (v).
\end{proof}

\begin{obs}
In \cite{Munn2000}, Munn investigates primeness of rings graded by inverse semigroups.
He shows (see \cite[Thm.~4.1]{Munn2000}) that if
$S$ is a so-called $0$-bisimple inverse semigroup,
$R$ is a faithful restricted $S$-graded ring, and $R_G$ is prime for some nonzero maximal subgroup $G$ of $S$,
then $R$ is prime.
We point out that Munn's theorem can potentially be used to prove e.g. (vi)$\Rightarrow$(i) in Theorem~\ref{P2-teo-nesgr}.
Indeed, we may associate a natural inverse semigroup $S(G):=G \cup \{z\}$ with the groupoid $G$ and view any $G$-graded ring as an $S(G)$-graded ring (see e.g. \cite[Sec.~4.3]{Lannstrom2022}).
It is easy to come up with examples of prime nearly-epsilon strongly $G$-graded rings
such that the corresponding $S(G)$-gradings fail to satisfy the requirements in Munn's theorem.
However, given a prime nearly epsilon-strongly $G$-graded ring $R$, it is not clear to the authors whether one can always find a subgroupoid $H$ of $G$, contained in $\Supp(R)$, such that
$S(H)$ and its grading on $R$ do in fact satisfy the requirements in Munn's theorem.
\end{obs}

We recall that Passman \cite{Passman1984} provided a characterization of prime unital strongly group graded rings. That result was generalized in \cite[Thm.~1.3]{Lannstrom2021} to nearly epsilon-strongly group graded rings. 

\begin{theorem}[{\cite[Thm.~1.3]{Lannstrom2021}\label{P2-nprimegroup}}] 
Let $G$ be a group and let $S$ be a nearly epsilon-strongly $G$-graded ring. The following statements are equivalent:
\begin{enumerate} [{\rm(a)}]
\item $S$ is not prime;
    
\item There exist:     \begin{enumerate} [{\rm(i)}]
    \item subgroups $N \lhd H \subseteq G$, 
    \item an $H$-invariant ideal $I$ of $S_e$ such that $I^gI=\{0\}$ for all $g \in G \setminus H,$ and
    \item nonzero ideals $\Tilde{A}, \Tilde{B}$ of $S_N$ such that $\Tilde{A}, \Tilde{B} \subseteq IS_N$ and $\Tilde{A}S_H\Tilde{B}=\{0\}.$ 
\end{enumerate}
\item There exist:    
\begin{enumerate} [{\rm(i)}]
    \item subgroups $N \lhd H \subseteq G$ with $N$ finite, 
    \item an $H$-invariant ideal $I$ of $S_e$ such that $I^gI=\{0\}$ for all $g \in G \setminus H,$ and
    \item nonzero ideals $\Tilde{A}, \Tilde{B}$ of $S_N$ such that $\Tilde{A}, \Tilde{B} \subseteq IS_N$ and $\Tilde{A}S_H\Tilde{B}=\{0\}.$ 
\end{enumerate}
\item There exist:    
\begin{enumerate} [{\rm(i)}]
    \item subgroups $N \lhd H \subseteq G$ with $N$ finite, 
    \item an $H$-invariant ideal $I$ of $S_e$ such that $I^gI=\{0\}$ for all $g \in G \setminus H,$ and
    \item nonzero $H$-invariant ideals $\Tilde{A}, \Tilde{B}$ of $S_N$ with $\Tilde{A}, \Tilde{B} \subseteq IS_N$ such that $\Tilde{A}S_H\Tilde{B}=\{0\}.$ 
\end{enumerate}
\item There exist:    
\begin{enumerate} [{\rm(i)}]
    \item subgroups $N \lhd H \subseteq G$ with $N$ finite, 
    \item an $H$-invariant ideal $I$ of $S_e$ such that $I^gI=\{0\}$ for all $g \in G \setminus H,$ and
    \item nonzero $H/N$-invariant ideals $\Tilde{A}, \Tilde{B}$ of $S_N$ such that $\Tilde{A}, \Tilde{B} \subseteq IS_N$ and $\Tilde{A}\Tilde{B}=\{0\}.$ 
\end{enumerate}
\end{enumerate}
\end{theorem}

\begin{obs}
Note that, in
Theorem~\ref{P2-teo-nesgr},
$\oplus_{g\in G_e^e} S_g$ is nearly epsilon-strongly graded by the group $G_e^e$.
Hence, one can use Theorem~\ref{P2-nprimegroup} to decide whether  $\oplus_{g\in G_e^e} S_g$ is prime.
\end{obs}

The following Theorem generalizes \cite[Thm.~1.4]{Lannstrom2021}.

\begin{theorem}\label{P2:nearly-torsion-free} 
Let $S$ be a nearly epsilon-strongly $G$-graded ring.
Suppose that there is some $e \in G_0'$ such that $G_e^e$ is torsion-free. Then $S$ is prime if, and only if, $S_e$ is $G_e^e$-prime and $\oplus_{e \in G_0}S_e$ is $G$-prime.
\end{theorem}

\begin{proof}
It follows from Theorem~\ref{P2-teo-nesgr} and \cite[Thm.~1.4]{Lannstrom2021}.
\end{proof}

\section{Applications to partial skew groupoid rings}\label{P2-Applications}

In this section, we will apply our main results on primeness for nearly epsilon-strongly groupoid graded rings to partial skew groupoid rings, (global) skew groupoid rings and groupoid rings. In particular, we will characterize prime partial skew groupoid rings induced by partial actions of group-type (cf.~\cite{Bagio2021}). Furthermore, we will generalize \cite[Thm.~12.4]{Lannstrom2021} and \cite[Thm.~13.7]{Lannstrom2021}. 

Throughout this section, unless stated otherwise, $G$ denotes an arbitrary groupoid and $A$ denotes an arbitrary ring. 

\subsection{Partial skew groupoid rings}

\begin{definition}
A \emph{partial action of a groupoid $G$ on a ring $A$} is a family of pairs $ \sigma =(A_g,\sigma_g)_{g \in G}$ satisfying:
\begin{enumerate}[{\rm (i)}]
\item for each $g \in G,$ $A_{r(g)}$
is an ideal of $A,$ $A_g$ is an ideal of $A_{r(g)}$ and $\sigma_g : A_{g^{-1}} \longrightarrow A_g$
is a ring isomorphism,
\item $\sigma_e =\id_{A_e}$, for every $e \in G_0$,
\item $\sigma_{h}^{-1}(A_{g^{-1}}\cap A_h)\subseteq A_{(gh)^{-1}}$, whenever $(g,h) \in G^2$,
\item $\sigma_g(\sigma_h(x)) = \sigma_{gh}(x)$, for all $x \in \sigma_h^{-1}(A_{g^{-1}}\cap A_h)$ and $(g,h) \in G^2$.
\end{enumerate}
\end{definition}

\begin{definition}
Given a partial action $\sigma$ of a groupoid $G$ on a ring $A$ one may define the \emph{partial skew groupoid ring $A \rtimes_{\sigma}G$} as the set of all formal sums of the form $\displaystyle\sum_{g \in G} a_g \delta_g$, where $a_g \in A_g$ is zero for all but finitely many $g\in G$, and with addition defined point-wise and multiplication given by
\begin{equation*}
a_g\delta_g \cdot b_h \delta_h =
\begin{cases}
 \alpha_g(\alpha_{g^{-1}}(a_g)b_h)\delta_{gh},& \text{ if } (g,h) \in G^2,\\
0, & \text{ otherwise. } 
\end{cases}
\end{equation*}
\end{definition}

\begin{obs}\label{P2:obs-partial-s-unital}
(a) 
Throughout this 
section,
unless stated otherwise,
will will assume that $\alpha=(A_g,\alpha_g)_{g\in G}$ is an arbitrary partial action of $G$ on $A$,
that $A_g$ is an $s$-unital ring, for every $g \in G$, and that $A=\oplus_{e\in G_0} A_e$.
As a consequence, 
$A\rtimes_{\sigma}G$ will always be an associative ring (see \cite[Rem.~2.7~(b)]{Moreira2022}), and there will exist a ring isomorphism (cf. \cite[Lem.~3.5~(i)]{Moreira2022}) $\psi: A \longrightarrow \oplus_{e \in G_0} A_e\delta_e$ defined by
\begin{align}\label{P2-map-psi}
    \psi \! \left(\sum_{e \in G_0} a_e\right) :=\sum_{e \in G_0} a_e \delta_e.
\end{align}

(b)
Under the above assumptions, by \cite[Lem.~3.2]{Moreira2022}, $A_g$ is an ideal of $A$, for every $g \in G.$

(c) 
It is readily verified that any partial skew groupoid ring $S:=A\rtimes_{\sigma}G$ carries a natural $G$-grading defined by letting $S_g:=A_g\delta_g$, for every $g\in G$.
\end{obs}

The following result generalizes \cite[Prop.~13.1]{Lannstrom2021}.
\begin{proposition}\label{P2-prop:SkewNearly}
The partial skew groupoid ring $A \rtimes_{\sigma}G$ is a nearly epsilon-strongly $G$-graded ring.
\end{proposition}

\begin{proof}
Let $g \in G$. Using that $A_{g^{-1}}$ is $s$-unital, and hence idempotent, we get that
$$(A_g\delta_g)(A_{g^{-1}}\delta_{g^{-1}})=\sigma_g(\sigma_{g^{-1}}(A_g)A_{g^{-1}})\delta_{r(g)}=\sigma_g(A_{g^{-1}}A_{g^{-1}})\delta_{r(g)}=\sigma_g(A_{g^{-1}})\delta_{r(g)}=A_g\delta_{r(g)}.$$
Now, using that $A_g$ is $s$-unital we get that $A_g \delta_{r(g)}$ is $s$-unital, and that $A_g$ is idempotent. Hence,
$$(A_g\delta_g)(A_{g^{-1}}\delta_{g^{-1}})(A_g\delta_g)=A_g\delta_{r(g)}A_g\delta_g=A_g^2\delta_g=A_g\delta_g.$$ This shows that $A \rtimes_{\sigma}G$ is  nearly epsilon-strongly $G$-graded.
\end{proof}

\begin{obs}
By Proposition~\ref{P2-prop:NearlyProperties} and Proposition~\ref{P2-prop:SkewNearly}, the partial skew groupoid ring $A \rtimes_{\sigma}G$ is $s$-unital.
\end{obs}

\begin{definition}\label{P2-def:partial-invariant}
Let $G$ be a groupoid, let $A$ be a ring and let $\sigma = (A_g, \sigma_g)_{g\in G}$ be a  partial action of $G$ on $A$. 
\begin{enumerate}[{\rm(i)}]
\item Let $H$ be a subgroupoid of $G.$ An ideal $I$ of $A$ is said to be \emph{$H$-invariant} if
$\sigma_g(I \cap A_{g^{-1}}) \subseteq I$ for every $g\in H$.

\item $A$ is said to be \emph{$G$-prime} if there are no nonzero $G$-invariant ideals $I,J$ of $A$ such that $IJ=\{0\}$.
\end{enumerate}
\end{definition}

The next result generalizes \cite[Rem.~13.4]{Lannstrom2021} from the group setting.

\begin{proposition}\label{P2-prop-corresp-inv-ideals}
$I$ is a $G$-invariant ideal of $A$ in the sense of Definition~\ref{P2-def:partial-invariant} if, and only if, $\psi(I)$ is a $G$-invariant ideal of $\oplus_{e \in G_0}A_e\delta_e$ in the sense of Definition~\ref{P2-def-invariant-nearly}. In particular, $A$ is $G$-prime if, and only if, $\oplus_{e \in G_0}A_e\delta_e$ is $G$-prime.
\end{proposition}

\begin{proof}
Suppose that $I$ is an ideal of $A$.
Let $g \in G.$ 
By
Remark~\ref{P2:obs-partial-s-unital}~(b)
and $s$-unitality of $A_g$, we get that $A_g I = I\cap A_g=IA_g$ and
$I A_g = 
A_g I A_g = A_g I$.
Furthermore,
$\psi(I)=\oplus_{e \in G_0}(IA_e)\delta_e.$ 
Notice that $A_g = A_{r(g)}A_g$.
We get that
$$\begin{aligned} \psi(I)^g & =A_{g^{-1}}\delta_{g^{-1}} \cdot \psi(I) \cdot A_g\delta_g = A_{g^{-1}}\delta_{g^{-1}} \cdot (IA_{r(g)})A_g\delta_g=A_{g^{-1}}\delta_{g^{-1}} \cdot IA_g\delta_g \\ & =\sigma_{g^{-1}}(\sigma_g(A_{g^{-1}})IA_g)\delta_{s(g)}=\sigma_{g^{-1}}(A_gIA_g)\delta_{s(g)}.\end{aligned}$$
Therefore, $$\psi(I)^g \subseteq \psi(I) \Longleftrightarrow \sigma_{g^{-1}}(A_gIA_g)\delta_{s(g)} \subseteq 
\psi(I) 
\Longleftrightarrow \sigma_{g^{-1}}(A_gIA_g) \subseteq I
\Longleftrightarrow \sigma_{g^{-1}}(I\cap A_g) \subseteq I.$$

\end{proof}

\begin{obs} \begin{enumerate}[{\rm(a)}]

\item Recall that, with the natural $G$-grading on $A\rtimes_\sigma G$,
an element $e \in G_0'$ is a \emph{support-hub} if for every nonzero element $a_g\delta_g$, with $g \in G$, there are $h,k \in G$ such that $s(h)=e,$ $r(k)=e$ and both $a_g\delta_gA_h\delta_h$ and $A_k\delta_ka_g\delta_g$ are nonzero.

\item 
For $e\in G_0$,
denote by $\sigma^e:=(A_h,\sigma_h)_{h\in G_e^e}$ the partial action 
of the isotropy group $G_e^e$ on the ring $A_e$, obtained by restricting $\sigma$.
The associated partial skew group ring is denoted by
 $A_e\rtimes_{\sigma^e} G_e^e$.
\end{enumerate}
\end{obs}

\begin{theorem}\label{P2-teo:partial-skew-prime} 
Let $\sigma = (A_g, \sigma_g)_{g\in G}$ be a partial action of $G$ on $A$ such that $A_g$ is $s$-unital for every $g \in G$ and  $A=\oplus_{e\in G_0} A_e$.  Then, the following statements are equivalent:
\begin{enumerate}[{\rm (i)}]
    \item The partial skew groupoid ring $A\rtimes_\sigma G$ is prime;
    \item $A$ is $G$-prime and, for every $e \in G_0',$ $A_e\rtimes_{\sigma^e} G_e^e$ is prime; 
    \item $A$ is $G$-prime and, for some $e \in G_0'$, $A_e\rtimes_{\sigma^e} G_e^e$ is prime; 
    \item $A\rtimes_\sigma G$ is graded prime and, for every $e \in G_0',$ $A_e\rtimes_{\sigma^e} G_e^e$ is prime; 
    \item $A\rtimes_\sigma G$ is graded prime and, for some $e \in G_0'$, $A_e\rtimes_{\sigma^e} G_e^e$ is prime; 
    
    \item For every $e \in G_0',$ $e$ is a support-hub and $A_e\rtimes_{\sigma^e} G_e^e$ is prime; 
    \item For some $e \in G_0',$  $e$ is a support-hub and $A_e\rtimes_{\sigma^e} G_e^e$ is prime.
\end{enumerate}
\end{theorem}

\begin{proof}
It follows from
Proposition~\ref{P2-prop:SkewNearly}, Theorem~\ref{P2-teo-nesgr}
and Proposition~\ref{P2-prop-corresp-inv-ideals}.
\end{proof}

We recall the following result from \cite{Lannstrom2021}. In that paper, the authors say that a partial skew group ring is \emph{$s$-unital} if it is defined by a partial group action on $s$-unital ideals.

\begin{theorem}[{\cite[Thm.~13.7]{Lannstrom2021}}]\label{P2:not-prime-13.7}
Let $G$ be a group and let $A \rtimes_{\sigma}G$ be an $s$-unital partial skew group ring. Then, $A \rtimes_{\sigma}G$ is not prime if, and only if, there are:

\begin{enumerate}[{\rm(i)}]
    \item subgroups $N \lhd H \subseteq G$ with $N$ finite, 
    \item an ideal $I$ of $A$ such that
    \begin{itemize}
        \item $\sigma_h(IA_{h^{-1}})=IA_h$ for every $h \in H,$ 
        \item $IA_g\cdot\sigma_g(IA_{g^{-1}})=\{0\}$ for every $g \in G\setminus H,$ and
    \end{itemize}
    \item nonzero ideals $\Tilde{B}, \Tilde{D}$ of $A\rtimes_\sigma N$ such that $\Tilde{B}, \Tilde{D} \subseteq I\delta_e(A \rtimes_\sigma N)$ and $\Tilde{B}\cdot A_h\delta_h \cdot \Tilde{D}=\{0\}$ for every $h \in H.$ 
\end{enumerate}    
\end{theorem}

\begin{obs}
Note that, in Theorem~\ref{P2-teo:partial-skew-prime}, $A_e\rtimes_{\sigma^e} G_e^e$ is an $s$-unital partial skew group ring.
Thus, one can apply Theorem~\ref{P2:not-prime-13.7} to determine whether  $A_e\rtimes_{\sigma^e} G_e^e$ is prime.
\end{obs}

\begin{definition}[{\cite[Rem.~3.4]{Bagio2021}}] \label{P2-def-group-type}
A partial action $\sigma = (A_g, \sigma_g)_{g\in G}$ of a connected groupoid $G$ on a ring $A$ is said to be of \emph{group-type} if there exist an element $e \in G_0$ and a family of morphisms $\{h_f\}_{f\in G_0}$ in $G$ such that $h_f:e \rightarrow f,$ $h_e=e,$ $A_{h_f^{-1}} = A_e$ and $A_{h_f}=A_f,$ for every $f \in G_0.$
\end{definition} 

\begin{obs}\label{P2:obs-group-type}
\begin{enumerate} [{\rm(a)}]
\item If a partial action $\sigma$ is of group-type, then every element of $G_0$ can take the role of $e$ in the above definition (see \cite[Rem.~3.4]{Bagio2021}).

\item By \cite[Lem.~3.1]{Bagio2021}, every global groupoid action is of group-type. The converse does not hold. 
For an example of a non-global partial action of group-type we refer the reader to \cite[Expl.~3.5]{Bagio2021}.
\end{enumerate}
\end{obs}

\begin{lema} \label{P2-lema:g.type-cond-support-hub} 
Let $\sigma = (A_g, \sigma_g)_{g\in G}$ be a  partial action of $G$ on $A$ such that $A_g$ is $s$-unital for every $g \in G$ and $A=\oplus_{e\in G_0} A_e.$ Consider the following statements:

\begin{enumerate} [{\rm(i)}]
    \item $\sigma$ is of group-type;
    \item There is some $e \in G_0'$ such that for every nonzero element $a_g\delta_g \in A\rtimes_{\sigma}G$ there is some $k \in G$ such that $s(k)=r(g),$ $r(k)=e$ and $a_g \in A_{k^{-1}};$
    \item There is some $e \in G_0'$ such that for every nonzero element $a_g\delta_g \in A\rtimes_{\sigma}G$ there is some $k \in G$ such that $s(k)=r(g),$ $r(k)=e$ and $A_{k^{-1}}a_g \neq \{0\};$
    \item There is some $e \in G_0'$ which is a support-hub. 
\end{enumerate}
Then, \rm{(i)} $\Rightarrow$ \rm{(ii)} $\Rightarrow$ \rm{(iii)} $\Leftrightarrow$ (iv).
\end{lema}

\begin{proof}
(i)$\Rightarrow$(ii) 
Suppose that (i) holds.
Let  $e \in G_0'.$ Suppose that $g \in G$ and $a_g\delta_g \neq 0.$ Since $\sigma$ is of group-type, there is a morphism $h_{r(g)}:e \rightarrow r(g)$ such that $A_{h_{r(g)}^{-1}}=A_e$ and $A_{h_{r(g)}}=A_{r(g)}.$ Note that $a_g \in A_g \subseteq A_{r(g)}=A_{h_{r(g)}}.$ Define $k:=h_{r(g)}^{-1}$ and the proof is done.

(ii)$\Rightarrow$(iii) 
Suppose that (ii) holds.
Let $e \in G_0'$ and $a_g\delta_g \in A\rtimes_{\sigma}G.$ By assumption, there is some $k \in G$ such that $s(k)=r(g),$ $r(k)=e$ and $a_g \in A_{k^{-1}}.$ Since $A_{k^{-1}}$ is $s$-unital, we get $A_{k^{-1}}a_g \neq \{0\}.$

(iii)$\Rightarrow$(iv) 
Suppose that (iii) holds.
Let $e \in G_0'$ be as in (iii).
Suppose that $g \in G$ and $a_g\delta_g \neq 0.$ 
By assumption, there is some $k \in G$ such that $s(k)=r(g)$ and $r(k)=e$ and $A_{k^{-1}}a_g \neq \{0\}.$ Then, there is some $d_{k^{-1}} \in A_{k^{-1}}$ such that $0 \neq d_{k^{-1}}a_g \in A_{k^{-1}}\cap A_g.$ Let $u \in A_{g^{-1}}$ be an $s$-unit for $\sigma_{g^{-1}}(d_{k^{-1}}a_g)$ and let $w \in A_{k^{-1}}$ be an $s$-unit for $d_{k^{-1}}a_g.$ Note that 
$$b:=((\sigma_k(w)\delta_{k})(d_{k^{-1}}\delta_{s(k)}))(a_g\delta_g)((u\delta_{g^{-1}})(w\delta_{k^{-1}})) \in A_k\delta_k(a_g\delta_g)A_{g^{-1}k^{-1}}\delta_{g^{-1}k^{-1}}.$$ 
And,
$$\begin{aligned} b & =
(\sigma_k(w)\delta_{k})(d_{k^{-1}}a_g\delta_g)(u\delta_{g^{-1}})(w\delta_{k^{-1}})
=(\sigma_k(w)\delta_{k})(\sigma_g(\sigma_{g^{-1}}(d_{k^{-1}}a_g)u)\delta_{gg^{-1}})(w\delta_{k^{-1}})\\
& =(\sigma_k(w)\delta_{k})(d_{k^{-1}}a_g\delta_{r(g)})(w\delta_{k^{-1}})
=(\sigma_k(w)\delta_{k})(d_{k^{-1}}a_g\delta_{k^{-1}})
= \sigma_k(\sigma_{k^{-1}}(\sigma_k(w))d_{k^{-1}}a_g)\delta_{kk^{-1}} \\
&= \sigma_k(d_{k^{-1}}a_g)\delta_e \neq 0.
\end{aligned}$$
Define $h:=g^{-1}k^{-1}$ and note that $s(g^{-1}k^{-1})=e$. Moreover, note that $A_k \delta_k a_g \delta_g$ and $a_g \delta_g A_h \delta_h$ are both nonzero.

(iv)$\Rightarrow$(iii) 
Suppose that (iv) holds.
Let $e \in G_0'$ be as in (iv) and let $a_g\delta_g$, with $g \in G$, be a nonzero element. By assumption, there is some $k \in G$ such that $r(k)=e$ and $A_k\delta_ka_g\delta_g$ is nonzero. Note that
$A_k\delta_ka_g\delta_g= \sigma_k(A_{k^{-1}}a_g)\delta_{kg}.$ Therefore, $A_{k^{-1}}a_g \neq\{0\}.$
\end{proof}

\begin{obs}
In general, statements (i) and (ii) in Lemma~\ref{P2-lema:g.type-cond-support-hub} are stronger than statements (iii) and (iv). This is shown by the following example.
\end{obs}

\begin{example}\label{P2:ex-partial-not-group-type}
We will construct a partial action of a groupoid on a ring, which is not of group-type, and such that the corresponding partial skew groupoid ring is prime.
Let $G:=\{e,f,g,g^{-1}\}$ be a groupoid with $G_0=\{e,f\},$ $s(g)=f$ and $r(g)=e$ as follows:

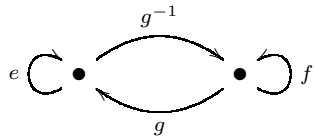
\[
\xymatrixcolsep{4pc}
\xymatrix{ {\bullet}\ar@/^1.2pc/[r]^{g^{-1}} \ar@(dl,ul)^{e} &
{\bullet}\ar@/^1.2pc/[l]^g \ar@(dr,ur)_{f} }
\]
Let $A:= \mathbb{Z}\oplus\mathbb{Z}$, and define the partial action $\sigma:=(A_g,\sigma_g)_{g \in G}$ of $G$ on $A,$ by
\begin{itemize} 
    \item $A_e:=\mathbb{Z}\oplus \{0\}$ and
    $A_f:=\{0\} \oplus \mathbb{Z};$
    \item $A_g:=2\mathbb{Z}\oplus \{0\}$ and $A_{g^{-1}}:=\{0\}\oplus 2 \mathbb{Z};$
    \item $\sigma_e:=\id_{\mathbb{Z}\oplus \{0\}}$ and $\sigma_f:=\id_{\{0\}\oplus \mathbb{Z}};$
    \item $\sigma_g (0,x) := (x-2,0)$ and $\sigma_{g^{-1}}(x,0) := (0,x+2).$
\end{itemize}
Notice that  $A_e\rtimes_{\sigma^e} G_e^e
\cong \mathbb{Z}$ and $A_f\rtimes_{\sigma^f} G_f^f
\cong \mathbb{Z}$ are prime rings. Moreover, $G=G'$ and $G$ is connected.

We claim that $A$ is $G$-prime. Let $I, J$ be nonzero $G$-invariant ideals of $A.$ Then, there are nonzero elements $x=(x_1,x_2) \in I$ and $y=(y_1,y_2) \in J.$ 
Without loss of generality, we may assume that $x_2 \neq 0$.
If $y_2 \neq 0,$ then $0 \neq xy \in IJ.$ Otherwise, $y_2 = 0$ in which case $\sigma_{g^{-1}}(y)=\sigma_{g^{-1}}(y_1,0)=(0,y_1+2) \in J$
and
$\sigma_{g^{-1}}(2y)=\sigma_{g^{-1}}(2y_1,0)=(0,2y_1+2) \in J$.
Therefore, 
$0 \neq x\sigma_{g^{-1}}(y) \in IJ$
or
$0 \neq x\sigma_{g^{-1}}(2y) \in IJ.$
By Theorem~\ref{P2-teo:partial-skew-prime}, $A \rtimes_{\sigma} G$ is prime.
In particular,
by Theorem~\ref{P2-teo:partial-skew-prime}, $e$ is a support-hub.

We claim that $e \in G_0$ does not satisfy condition (ii) in Lemma~\ref{P2-lema:g.type-cond-support-hub}. Observe that $(0,3)\delta_f \in A_f\delta_f$ and that $g$ is the only element of $G$ such that $s(g)=f$ and $r(g)=e.$ However, $(0,3) \notin A_{g^{-1}}.$ A similar argument shows that $f \in G_0$ also does not satisfy statement (ii) in Lemma~\ref{P2-lema:g.type-cond-support-hub}. In particular, $\sigma$ is not of group-type. 
\end{example}

\begin{theorem}\label{P2-prop:g-type-prime} 
Let $\sigma = (A_g, \sigma_g)_{g\in G}$ be a  partial action of $G$ on $A$ such that $A_g$ is $s$-unital for every $g \in G,$  $A=\oplus_{e\in G_0} A_e$ and  $\sigma$ is of group-type. Then the partial skew groupoid ring $A\rtimes_{\sigma}G$ is prime if, and only if, there is some $e \in G_0'$ such that $A_e\rtimes_{\sigma^e} G_e^e$ is prime.
\end{theorem}

\begin{proof} 
It follows from Lemma~\ref{P2-lema:g.type-cond-support-hub} and Theorem~\ref{P2-teo:partial-skew-prime}.
\end{proof}

Now, we will make use of the example from \cite[Expl.~3.5]{Bagio2021} of a non-global partial action of a groupoid on a ring, %
of group-type,
and apply Theorem~\ref{P2-prop:g-type-prime} to it. 

\begin{example}
Let $G=\{e,f,g,h,l,m,l^{-1},m^{-1}\}$ be the groupoid with $G_0=\{e,f\}$ and the following composition rules:
$$ g^2=e, \hspace{0.5cm} h^2= f, \hspace{0.5cm} lg=m=hl, \hspace{0.5cm} g \in G_e^e, \hspace{0.5cm} h \in G_f^f \hspace{0.4cm} \text{ and } \hspace{0.4cm} l,m: e \rightarrow f. $$ 
We present in the following diagram the structure of $G$:

\[
\xymatrixcolsep{4pc}
 \xymatrix{& e \ar[r]^{l}  &  f  \ar[d]^{h} \\
&e \ar[u]^{g} \ar[r]_{m} & f}
\]

Let $\mathbb{C}$ be the field of complex numbers and let $A=\mathbb{C}e_1 \oplus \mathbb{C} e_2 \oplus \mathbb{C} e_3 \oplus \mathbb{C} e_4,$ where $e_ie_j=\delta_{i,j}e_i$ and $e_1+\ldots+e_4=1.$ We define the partial action $(A_t,\sigma_t)_{t \in G}$ of $G$ on $A$ as follows:
$$\begin{aligned} & A_e = \mathbb{C}e_1 \oplus \mathbb{C} e_2 = A_{l^{-1}}, \hspace{0.6cm} A_f= \mathbb{C}e_3 \oplus \mathbb{C} e_4=A_l, \\ 
& A_g = \mathbb{C}e_1 =A_{g^{-1}}=A_{m^{-1}}, \hspace{0.6cm} A_m=A_h=\mathbb{C}e_3=A_{h^{-1}}, 
\end{aligned}$$
and
$$\begin{aligned} & \hspace{0.6cm} \sigma_e= \id_{A_e}, \hspace{0.6cm} \sigma_f= \id_{A_f}, \hspace{0.6cm} \sigma_g: a e_1 \mapsto \overline{a}e_1, \hspace{0.6cm} \sigma_h: a e_3 \mapsto \overline{a}e_3, \hspace{0.6cm} \sigma_m: a e_1 \mapsto \overline{a}e_3, \\
& \sigma_{m^{-1}}: a e_3 \mapsto \overline{a}e_1, \hspace{0.6cm} \sigma_l: a e_1 + b e_2 \mapsto ae_3+ b e_4, \hspace{0.6cm} \sigma_{l^{-1}}: a e_3 + b e_4 \mapsto ae_1 +b e_2,
\end{aligned}$$
where $\overline{a}$ denotes the complex conjugate of $a,$ for all $a \in \mathbb{C}.$ By choosing $h_e:=e$ and $h_f:=l,$ we notice that $\sigma$ is of group-type (cf. Definition~\ref{P2-def-group-type}). 

Now, we describe the group partial action $\sigma^e=(A_t,\sigma_t)_{t \in G_e^e}$ of $G_e^e$ on $A_e.$ Note that $$G_e^e=\{e,g\}, \hspace{0.6cm} \sigma_e= \id_{A_e} \hspace{0.5cm} \text{ and } \hspace{0.5cm} \sigma_g: a e_1 \mapsto \overline{a}e_1.$$
We claim that $A_e$ is not $G_e^e$-prime. Let $I:=\mathbb{C}e_1 \oplus \{0\}e_2$ and $J:=\{0\}e_1\oplus \mathbb{C}e_2.$ Note that $I$ and $J$ are nonzero $G_e^e$-invariant ideals of $A_e$ and $IJ=\{0\}.$ By Theorem~\ref{P2-teo:partial-skew-prime}, $A_e\rtimes_{\sigma^e}G_e^e$ is not prime. An analogous argument shows that $A_f\rtimes_{\sigma^f}G_f^f$ is not prime.
Hence, Theorem~\ref{P2-prop:g-type-prime} implies that $A\rtimes_{\sigma}G$ is not prime.
\end{example}

The following result generalizes \cite[Thm.~13.5]{Lannstrom2021} from the group setting.

\begin{theorem}\label{P2-teo-torsion-partial} 
Let $\sigma = (A_g, \sigma_g)_{g\in G}$ be a  partial action of $G$ on $A$ such that $A_g$ is $s$-unital for every $g \in G,$ and $A=\oplus_{e\in G_0} A_e.$ 
Furthermore, suppose that there is some $e \in G_0'$ such that $G_e^e$ is torsion-free. Then $A\rtimes_{\sigma}G$ is prime if, and only if, $A_e$ is $G_e^e$-prime and $A$ is $G$-prime.
\end{theorem}

\begin{proof}
It follows from Proposition~\ref{P2-prop:SkewNearly}, Proposition~\ref{P2-prop-corresp-inv-ideals} and Theorem~\ref{P2:nearly-torsion-free}.
\end{proof}

\subsection{Skew groupoid rings}

The partial action $\sigma =(A_g,\sigma_g)_{g \in G}$ of $G$ on $A$ is said to be \emph{global} if $A_g=A_{r(g)}$ for every $g\in G$.
In that case, the corresponding partial skew groupoid ring $A \rtimes_\sigma G$ is said to be a \emph{skew groupoid ring} (see e.g. \cite{OL2010,OL2012}).

\begin{obs}
Let $\sigma =(A_g,\sigma_g)_{g \in G}$ be a global action of $G$ on $A$.

(a) For $g\in G$, $\sigma_g : A_{s(g)} \to A_{r(g)}$ is a ring isomorphism.

(b) Note that $\sigma_g \circ \sigma_h = \sigma_{gh}$, whenever $(g,h)\in G^{(2)}$.

(c) The multiplication rule on the skew groupoid ring is induced by the following somewhat simplified rule compared to the partial case: 
\begin{displaymath}
(a \delta_g)(b \delta_h) :=
\left\{
\begin{array}{ll}
     a \sigma_g(b) \delta_{gh}, & \text{ if } (g,h) \in G^{(2)}\\
     0, & \text{ otherwise}
\end{array}\right.
\end{displaymath}
for $g,h\in G$ and $a \in A_{r(g)}$, $b\in A_{r(h)}$.
\end{obs}

\begin{theorem}\label{P2:teo-global-connected} 
 Let $\sigma = (A_g, \sigma_g)_{g\in G}$ be a global action of $G$ on $A$ such that $A_e$ is $s$-unital for every $e\in G_0$,
 and let $A=\oplus_{e\in G_0} A_e$. 
Suppose that the groupoid $G$ is connected.
Then, the skew groupoid ring $A\rtimes_{\sigma}G$ is prime if, and only if, there is some $e \in G_0'$ such that $A_e\rtimes_{\sigma^e} G_e^e$ is prime.
\end{theorem}

\begin{proof}
It follows from Remark~\ref{P2:obs-group-type} (b) and Theorem~\ref{P2-prop:g-type-prime}.
\end{proof}

\begin{proposition}\label{P2-prop:support-hub-G'connected}
Let $\sigma = (A_g, \sigma_g)_{g\in G}$ be a global action of $G$ on $A$ such that $A_e$ is $s$-unital for every $e\in G_0$,
 and let $A=\oplus_{e\in G_0} A_e$.
 The following statements are equivalent:
\begin{enumerate} [{\rm(i)}]
    \item $G'$ is connected;
    \item For every $e \in G_0',$ $e$ is a support-hub;
    \item For some $e \in G_0',$ $e$ is a support-hub.
\end{enumerate}
\end{proposition}

\begin{proof} 
Obviously, (ii)$\Rightarrow$(iii). 

(iii)$\Rightarrow$(i) follows from Proposition~\ref{P2:prop-connected} (ii).

(i)$\Rightarrow$(ii) 
Suppose that $G'$ is connected.
Take $e \in G_0'$, $g\in G$,  and let $a_g\delta_g \in A\rtimes_{\sigma}G$ be a nonzero element. Then $g \in G'$ and, by assumption, there is some $k \in G'$ such that $s(k)=r(g)$ and $r(k)=e.$ Notice that $a_g \in A_{r(g)}=A_{s(k)}=A_{k^{-1}}.$ By Lemma~\ref{P2-lema:g.type-cond-support-hub} (ii) $\Rightarrow$ (iv), $e$ is a support-hub.
\end{proof}

Below, we summarize our findings for skew groupoid ring.

\begin{theorem}\label{P2-teoskew}
Let $\sigma = (A_g, \sigma_g)_{g\in G}$ be a global action of $G$ on $A$ such that $A_e$ is $s$-unital for every $e\in G_0$,
 and let $A=\oplus_{e\in G_0} A_e$.
Then, the following statements are equivalent:
\begin{enumerate}[{\rm (i)}]
    \item The skew groupoid ring $A\rtimes_\sigma G$ is prime;
    \item $A$ is $G$-prime, and for every $e \in G_0',$ $A_e\rtimes_{\sigma^e} G_e^e$ is prime; 
    \item $A$ is $G$-prime, and for some $e \in G_0',$ $A_e\rtimes_{\sigma^e} G_e^e$ is prime; 
    \item $A\rtimes_\sigma G$ is graded prime, and for every $e \in G_0',$ $A_e\rtimes_{\sigma^e} G_e^e$ is prime; 
    \item $A\rtimes_\sigma G$ is graded prime, and for some $e \in G_0',$ $A_e\rtimes_{\sigma^e} G_e^e$ is prime;
    \item $G'$ is connected, and for every $e \in G_0',$ $A_e\rtimes_{\sigma^e} G_e^e$ is prime;
    \item $G'$ is connected, and for some $e \in G_0'$,  $A_e\rtimes_{\sigma^e} G_e^e$ is prime.
\end{enumerate}
\end{theorem}

\begin{proof}
It follows from Theorem~\ref{P2-teo:partial-skew-prime} and Proposition~\ref{P2-prop:support-hub-G'connected}.
\end{proof}

The following example shows that Theorem~\ref{P2-teoskew} does not generalize to partial skew groupoid rings.

\begin{example}\label{P2:ex-partial-connected-not-prime}
Let $G=\{e,f,g,g^{-1}\}$ be a groupoid such that $G_0=\{e,f\},$ $s(g)=f$ and $r(g)=e$ as follows:
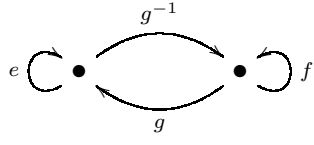
\[
\xymatrixcolsep{4pc}
\xymatrix{ {\bullet}\ar@/^1.2pc/[r]^{g^{-1}} \ar@(dl,ul)^{e} &
{\bullet}\ar@/^1.2pc/[l]^g \ar@(dr,ur)_{f} }
\]

Let $A:= \mathbb{Z}\oplus\mathbb{Z}.$ Now, we define a partial action $\sigma:=(A_g,\sigma_g)_{g \in G}$ of $G$ on $A$:
\begin{itemize} 
    \item $A_e=\mathbb{Z}\oplus \{0\}$ and
    $A_f=\{0\} \oplus \mathbb{Z};$
    \item $A_g=A_{g^{-1}}=\{0\};$
    \item $\sigma_e=\id_{\mathbb{Z}\oplus \{0\}},$ $\sigma_f=\id_{\{0\}\oplus \mathbb{Z}}$ and $\sigma_g=\sigma_{g^{-1}}=\id_{\{0\}}.$
\end{itemize}
Notice that $A_e\rtimes_{\sigma^e} G_e^e 
\cong \mathbb{Z}$ and $A_f\rtimes_{\sigma^f} G_f^f
\cong \mathbb{Z}$ are prime rings. Observe that $G=G'$, and that $G$ is connected. However, there are no $t \in G$ and $c_t\delta_t \in A\rtimes_\sigma G$ such that  $\delta_e(c_t\delta_t)\delta_f\neq 0$. Hence, by Lemma~\ref{P2-lemmagraded1}, $A\rtimes_\sigma G$ is not graded prime. 
\end{example}

\begin{corolario}\label{P2:global-torsion-free} 
Let $\sigma = (A_g, \sigma_g)_{g\in G}$ be a global action of $G$ on $A$ such that $A_e$ is $s$-unital for every $e\in G_0$,
 and let $A=\oplus_{e\in G_0} A_e$.
Suppose that there is some $e \in G_0'$ such that $G_e^e$ is torsion-free. Then, the skew groupoid ring $A\rtimes_{\sigma}G$ is prime if, and only if, $A_e$ is $G_e^e$-prime and $G'$ is connected.
\end{corolario}

\begin{proof}
It follows from Theorem~\ref{P2-teoskew} and \cite[Thm.~13.5]{Lannstrom2021}.
\end{proof}

\subsection{Groupoid rings}

Let $R$ be an $s$-unital ring and let $G$ be a groupoid. The groupoid ring $R[G]$ consists of elements of the form $\sum_{g\in G} a_g \delta_g$ where $a_g\in R$ is zero for all but finitely many $g\in G$. 
For $g,h\in G$ and $a,b\in R$,
the multiplication in $R[G]$ is defined by the relation
$a \delta_g \cdot b \delta_h := ab \delta_{gh}$,
if $g,h$ are composable, and $a \delta_g \cdot b \delta_h :=0$ otherwise.

\begin{obs}\label{P2-obs:GroupoidRingAsGlobalCrossedProduct}
Let $R$ be an $s$-unital ring and let $G$ be a groupoid. 
Consider the global action $\sigma=(A_g,\sigma_g)_{g\in G}$ of $G$ on $A$, defined by letting
$A_g:=R$ and $\sigma_g := \id_R$ for every $g\in G$, and $A:=\oplus_{e\in G_0} A_e=\oplus_{e\in G_0} R$.
Notice that the corresponding skew groupoid ring $A \rtimes_\sigma G$ is isomorphic to the groupoid ring $R[G]$.
\end{obs}

A subset $X \subseteq G_0$ is said to be \emph{$R$-dense} if for every nonzero $a \in R[G]$ there is some $g \in \Supp(a)$ such that $s(g) \in X.$ For each $e \in G_0,$ define $\mathcal{O}_e:=\{f \in G_0 : \exists g \in G, s(g)=e, r(g) = f \}$.

\begin{proposition}\label{P2-PropAdense}
Let $G$ be a groupoid and let $R$ be a nonzero $s$-unital ring. The following statements are equivalent:
\begin{enumerate} [{\rm(i)}]
\item For every $e \in G_0,$  $\mathcal{O}_e$ is $R$-dense;
\item There is some $e \in G_0$ such that $\mathcal{O}_e$ is $R$-dense;
\item $G$ is connected.
\end{enumerate}    
\end{proposition}

\begin{proof}
(i)$\Rightarrow$(ii) The proof is immediate. 

(ii)$\Rightarrow$(iii) 
Suppose that there is some $e \in G_0$ such that $\mathcal{O}_e$ is $R$-dense. Let $f,h \in G_0$ and let $r\in R$ be nonzero. Clearly, $r\delta_f, r\delta_h \in R[G],$ and hence, by assumption, $f,h \in \mathcal{O}_e.$ 
By the definition of $\mathcal{O}_e$, we may find some $g\in G$ such that $s(g)=f$ and $r(g)=h.$

(iii)$\Rightarrow$(i) Fix some $e \in G_0$ and suppose that $G$ is connected.
Clearly, $\mathcal{O}_e=G_0$ which is $R$-dense. 
\end{proof}

The following theorem generalizes \cite[Thm.~12.4]{Lannstrom2021}. 

\begin{theorem}\label{P2-teogr}
Let $G$ be a groupoid and let $R$ be a nonzero $s$-unital ring. The following statements are equivalent:
\begin{enumerate} [{\rm(i)}]
    \item The groupoid ring $R[G]$ is prime;
    \item $G$ is connected and there is some $e \in G_0$ such that the group ring $R[G_e^e]$ is prime;
    \item $G$ is connected and, for every $e \in G_0,$ the group ring $R[G_e^e]$ is prime;
    \item There is some $e \in G_0$ such that $\mathcal{O}_e$ is $R$-dense and $R[G_e^e]$ is prime;
    \item For every $e \in G_0,$  $\mathcal{O}_e$ is $R$-dense and $R[G_e^e]$ is prime;
    \item $G$ is connected, $R$ is prime and there is some $e \in G_0$ such that $G_e^e$ has no non-trivial finite normal subgroup;
    \item $G$ is connected, $R$ is prime and, for every $e \in G_0,$ $G_e^e$ has no non-trivial finite normal subgroup;
    \item $R$ is prime, and there is some $e \in G_0$ such that $\mathcal{O}_e$ is $R$-dense, and $G_e^e$ has no non-trivial finite normal
    subgroup;
    \item $R$ is prime, and for every $e \in G_0,$ $\mathcal{O}_e$ is $R$-dense, and $G_e^e$ has no non-trivial finite normal
    subgroup.
\end{enumerate}
\end{theorem}

\begin{proof} 
Notice that $G'=G.$ The proof follows from Remark~\ref{P2-obs:GroupoidRingAsGlobalCrossedProduct}, Theorem~\ref{P2-teoskew}~(i), (vi) and (vii), Proposition~\ref{P2-PropAdense} and \cite[Thm.~12.4]{Lannstrom2021}. 
\end{proof}

\begin{obs} (a) It is known that, in the case where $R$ is a commutative unital ring, the groupoid ring $R[G]$ is an example of a Steinberg algebra
(see \cite[Rem.~4.10]{Steinberg2010}). Hence, in that special case, the equivalence between (i) and (iv) in Theorem~\ref{P2-teogr} can be obtained using Steinberg's results from \cite[Prop.~4.3]{Steinberg2019}, \cite[Prop.~4.4]{Steinberg2019} and \cite[Thm.~4.9]{Steinberg2019}.

(b) In the case where $R$ is unital, after suitable translations of the properties involved, it is possible to obtain e.g. the implication (ii)$\Rightarrow$(i) in Theorem~\ref{P2-teogr} from \cite[Thm.~3.2]{Munn1990}. 
\end{obs}

\end{document}